\documentclass[reqno,11pt]{amsart}
\usepackage{amsmath,amssymb,latexsym,soul,cite,mathrsfs}
\usepackage{color,enumitem,graphicx}
\usepackage[colorlinks=true,urlcolor=blue,
citecolor=red,linkcolor=blue,linktocpage,pdfpagelabels,
bookmarksnumbered,bookmarksopen]{hyperref}
\usepackage[english]{babel}
\usepackage[left=2.6cm,right=2.6cm,top=2.9cm,bottom=2.9cm]{geometry}
\usepackage[hyperpageref]{backref}
\newtheorem{theorem}{Theorem}
\newtheorem{lemma}[theorem]{Lemma}
\newtheorem{corollary}[theorem]{Corollary}
\newtheorem{proposition}[theorem]{Proposition}
\newtheorem{remark}[theorem]{Remark}
\newtheorem{definition}[theorem]{Definition}
\newtheorem{conjecture}[theorem]{Conjecture}

\newtheorem{theoremletter}{Theorem}

\newenvironment{acknowledgement}{\noindent\textbf{Acknowledgments.}}{}

\newcommand{\innerthmname}{}

\theoremstyle{definition}

\makeatletter
\def\namedlabel#1#2{\begingroup
	#2%
	\def\@currentlabel{#2}%
	\phantomsection\label{#1}\endgroup
}
\makeatother

\newcommand{\ud}{\mathrm{d}}

\title[Classification for solutions to critical sixth order equations]{Classification for positive singular solutions to critical sixth order equations} 
\thanks{This research is partially supported by S\~ao Paulo Research Foundation (FAPESP) \#2020/07566-3 and \#2021/15139-0 and Natural Sciences and Engineering Research Council of Canada (NSERC)}

\author[J.H. Andrade]{Jo\~{a}o Henrique Andrade}
\author[J. Wei]{Juncheng Wei}

\address[J.H. Andrade]{
	Department of Mathematics,
	University of British Columbia
	\newline\indent 
	V6T 1Z2, Vancouver-BC, Canada
	\newline\indent
	and
	\newline\indent 
	Institute of Mathematics and Statistics,
	University of S\~ao Paulo
	\newline\indent 
	05508-090, S\~ao Paulo-SP, Brazil
}
\email{\href{mailto:andradejh@math.ubc.ca}{andradejh@math.ubc.ca}}
\email{\href{mailto:andradejh@ime.usp.br}{andradejh@ime.usp.br}}

\address[J. Wei]{
	Department of Mathematics,
	University of British Columbia
	\newline\indent 
	V6T 1Z2, Vancouver-BC, Canada}
\email{\href{mailto:jcwei@math.ubc.ca}{jcwei@math.ubc.ca}}


\begin{document}
	
	\begin{abstract}
		We classify entire positive singular solutions to a family of critical sixth order equations in the punctured space with a non-removable singularity at the origin. 
		More precisely, we show that when the origin is a non-removable singularity, solutions are given by a singular radial factor times a periodic solution to a sixth order IVP with constant coefficients.
		On the technical level, we combine integral sliding methods and qualitative analysis of ODEs, based on a conservation of energy result, to perform a topological two-parameter shooting technique.
		We first use the integral representation of our equation to run a moving spheres technique, which proves that solutions are radially symmetric with respect to the origin.
		Thus, in Emden--Fowler coordinates, we can reduce our problem to the study of an sixth order autonomous ODE with constant coefficients.
		The main heuristics behind our arguments is that since all the indicial roots of the ODE operator are positive, it can be decomposed into the composition of three second order operators satisfying a comparison principle. 
		This allows us to define a Hamiltonian energy which is conserved along solutions, from which we extract their qualitative properties, such as uniqueness, boundedness, asymptotic behavior, and classification.
	\end{abstract}
	
	\maketitle
	
	\bigskip
	\begin{center}
		\footnotesize
		\tableofcontents
	\end{center}
	
	\section{Introduction}
	We are interested in classifying (classical) positive singular solutions $u\in C^{6}(\mathbb{R}^n\setminus\{0\})$ with $n\geqslant7$ (which will always be assumed so forth) to the following family of  critical sixth order PDEs on the punctured space
	\begin{flalign}\tag{$\mathcal P_{6,\infty}$}\label{ourPDE}
		(-\Delta)^3u=c_{n}u^{\frac{n+6}{n-6}} \quad {\rm in} \quad \mathbb{R}^n\setminus\{0\}.
	\end{flalign}
	Here $\Delta^3=\Delta\circ\Delta\circ\Delta$ is the tri-Laplacian and $f(u):=c_{n}|u|^{p-2}u$ with $p=\frac{2n}{n-6}:=2^{\#}$ is critical in the sense of the compact Sobolev embedding of $H^{3}(\mathbb R^n)$ and $c_{n}$ is a normalizing geometric constant given by
	\begin{equation*}
		c_{n}:=\frac{n(n-6)(n^4-20n^2+64)}{64}.
	\end{equation*}
	
	We say that a positive solution $u\in C^{6}(\mathbb{R}^n\setminus\{0\})$ has a removable singularity at the origin if $\lim_{x\rightarrow0}u(x)<\infty$, that is, it can be continuously extended across the origin. 
	Otherwise, we say that the origin is a non-removable singularity. 
	Let us call these solutions non-singular and singular, respectively. 
	
	Let us mention that the second-named author and X. Xu \cite{MR1679783} studied non-singular solutions to \eqref{ourPDE}. 
	On this subject, they are based on a sliding technique to prove that all its positive non-singular  solutions are radially symmetric with respect to some point and have a closed expression.
	This result confirmed a conjecture by E. Lieb \cite{MR717827} about the classification of extremal function for the Sobolev inequality. These results can be stated as follows
	
	\begin{theoremletter}\label{thm:chen-lin-ou}
		Let $u\in C^{6}(\mathbb{R}^n)$ be a positive non-singular   solution to \eqref{ourPDE}. Then, there exist $x_0\in\mathbb R^n$ and $\mu\in\mathbb R$ such that $u$ is radially symmetric with respect to $x_0$ and monotonically decreasing. Moreover, it holds
		\begin{equation*}
			u(x)=\left(\frac{2\mu}{\mu^2+|x-x_0|^2}\right)^{\frac{6-n}{2}}.
		\end{equation*}
		The elements in this family are called spherical solutions and they are denoted by $u_{x_0,\mu}$.
	\end{theoremletter}
	
	Our main result in this manuscript extends Theorem~\ref{thm:chen-lin-ou} for the case when the origin is a non-removable singularity and can be stated as 
	\begin{theorem}\label{maintheorem}
		Let $u\in C^{6}(\mathbb{R}^n\setminus\{0\})$ be a positive singular solution to \eqref{ourPDE}. Then, $u$ is radially symmetric with respect to the origin and monotonically decreasing. Moreover, there exist $a_0\in\mathbb (0,a^{*}_n)$ and $T\in[0,T_{a_0}]$ such that
		\begin{equation*}
			u(x)=|x|^{\frac{6-n}{2}}v_{a}(-\ln |x|+T).
		\end{equation*}
		Here $v_{a}$ is the unique periodic bounded positive solution to the sixth order IVP $($or Cauchy problem$)$ below
		\begin{flalign}\tag{$\mathcal O_{6,\infty}$}\label{ourIVP}
			\begin{cases}
				v^{(6)}-K_4v^{(4)}+K_2v^{(2)}-K_0v+c_nv^{\frac{n+6}{n-6}}=0 \quad {\rm in} \quad \mathbb{R}\\
				v(0)=a_0, \quad v^{(2)}(0)=a_2, \quad v^{(4)}(0)=a_4, \quad {\rm and} \quad v^{(1)}(0)=v^{(3)}(0)=v^{(5)}(0)=0.
			\end{cases}
		\end{flalign}
		Here $K_0,K_2,K_4$ are dimensional constants $($see \eqref{coefficients}$)$, $a_2,a_4$ depend on $a_0$, and $a_*=K_0^{(n-6)/12}$.
		The elements in this family are called $($sixth order$)$ Emden--Fowler solutions and are denoted by $u_{a,T}$.
	\end{theorem}
	
	\begin{remark}
		Notice that the assumption that our solutions are classical is not restrictive.
		In fact, by a standard regularity lifting technique from \cite{MR1338474}, one can show that any weak solution $u\in H^3(\mathbb R^n\setminus\{0\})$ to \eqref{ourPDE} also salves \eqref{ourPDE} in classical sense, that is, $u\in C^6(\mathbb R^n\setminus\{0\})$ or even smooth $u\in C^\infty(\mathbb R^n\setminus\{0\})$ and \eqref{ourPDE} holds pointwise.
	\end{remark}
	
	Recently, T. Jin and J. Xiong \cite{arxiv:1901.01678} used variational techniques to prove existence of solutions to \eqref{ourPDE}.
	In fact, they considered a integral equation generalizing \eqref{ourPDE}, which includes non-local versions of the nonlinearity (see also \cite{MR2200258}).
	The full classification result for the critical fractional case is still unknown. 
	For other ranges of the power-nonlinearity, S. Luo et al. \cite{MR4257807,10.1093/imrn/rnab212} proved a monotonicity formula for solutions to related higher order poly-harmonic equations.
	We refer the reader to \cite{MR3850774,MR3611024,MR3747434,MR3636636,MR3190428,MR3973901,MR4123335,MR3632218,MR4438901} for more results on this subject. 
	
	Eq. \eqref{ourPDE} arises naturally in the study of the lack of compactness for the critical Sobolev embedding. 
	Besides, its applicability in PDEs, our result has also a connection with differential geometry.
	In this language, \eqref{ourPDE} is the locally conformally flat version of a more general geometric equation, the so-called sixth order GJMS equation \cite{MR2858236,MR1190438}, which is driven by the $P^6$ the sixth order GJMS operator. 
	In this geometrical language, our main results classify the metrics with constant (sixth order) $Q^6$-curvature on $(\mathbb S^n\setminus\{p,-p\},g_0)$, where $g_0$ is the standard round metric on the punctured unit sphere and
	$Q^6_{g_0}\equiv 2^{-5}{n(n^4-20n^2+64)}$.
	Indeed, the metric $g=u^{6/{(n-6)}}g_0$ has constant $Q^6$-curvature if, and only, if $u\in C^{\infty}(\mathbb{S}^n\setminus\{0\})$ is a (smooth) positive solution to the geometric PDE
	\begin{equation*}
		P^6_{g_0}u=c_nu^{\frac{n+6}{n-6}} \quad {\rm in} \quad \mathbb S^n\setminus\{p,-p\},
	\end{equation*}
	where
	\begin{equation*}
		P_{g_0}^6=\left(-\Delta_{g_0}+\frac{(n-6)(n+4)}{4}\right)\left(-\Delta_{g_0}+\frac{(n-4)(n+2)}{4}\right)\left(-\Delta_{g_0}+\frac{n(n-2)}{4}\right).
	\end{equation*}
	For more details on this subject, we refer the interested reader to \cite{MR3077914,MR3652455,MR4285731,MR3073887} and the references therein.
	It is also worth mentioning the relations of \eqref{ourPDE} with tri-harmonic maps \cite{MR3624536}.
	
	\begin{remark}
		In \cite{arxiv:1901.01678}, the asymptotic behavior of solutions to a local equation on the punctured ball of radius $R>0$ is studied.
		More precisely, it is proved that solutions to 
		\begin{flalign}\tag{$\mathcal P_{6,R}$}\label{ourlocalPDE}
			(-\Delta)^3u=c_{n}u^{\frac{n+6}{n-6}} \quad {\rm in} \quad B_R\setminus\{0\},
		\end{flalign}
		satisfy
		\begin{equation*}
			u(x)=(1+\mathrm{o}(1))u_0(x) \quad {\rm as} \quad x\rightarrow0,
		\end{equation*}
		where $u_0\in C^6(\mathbb R^n\setminus\{0\})$ solves \eqref{ourPDE}.
		Our main theorem may be applied to improve the asymptotic behavior of solutions to \eqref{ourlocalPDE} near the origin.
		Based on \cite{MR1666838,arXiv:2003.03487,arxiv.2001.07984}, we conjecture that assuming $-\Delta u>0$ and $\Delta^2u>0$, one can find $\beta_0>0$ such that
		\begin{equation*}
			u(x)=(1+\mathcal{O}(|x|^{\beta_0}))u_0(x) \quad {\rm as} \quad x\rightarrow0.
		\end{equation*}
		A result like this have direct implications on the study of the $($sixth order$)$ singular GJMS.
	\end{remark}
	
	Now let us compare our results to the ones in the fourth and second order cases. 
	First, we consider positive solutions $u\in C^{4}(\mathbb{R}^n\setminus\{0\})$ the fourth order critical equation
	\begin{flalign}\tag{$\mathcal P_{4,\infty}$}\label{ourPDE4th}
		(-\Delta)^2u=\frac{n(n-4)^2(n^2-4)}{16}u^{\frac{n+4}{n-4}} \quad {\rm in} \quad \mathbb{R}^n\setminus\{0\},
	\end{flalign}
	where $n\geqslant 5$, $\Delta^2=\Delta\circ\Delta$ is the bi-Laplacian. 
	Notice that \eqref{ourPDE4th} is critical in the sense of the compact Sobolev embedding $H^{2}(\mathbb R^n)$.
	On this subject, we should mention that when the origin is a non-removable singularity, C. S. Lin \cite{MR1611691} (see also \cite{MR1769247}) obtained radial symmetry for positive solutions to \eqref{ourPDE4th} using the asymptotic moving planes technique. Recently, Z. Guo, et al. \cite{MR4094467} proved the existence of periodic solutions by applying a mountain pass theorem and conjectured that all ODE solutions should be periodic.
	Later on, R. L. Frank and T. K\"onig \cite{MR3869387}  answered this conjecture, obtaining more accurate results concerning the classification for entire positive singular solutions to \eqref{ourPDE4th}.
	
	\begin{theoremletter}
		Let $u\in C^{4}(\mathbb{R}^n\setminus\{0\})$ be a positive  solution to \eqref{ourPDE4th}. Then, $u$ is radially symmetric with respect to the origin and monotonically decreasing. Moreover,
		\begin{itemize}
			\item[{\rm (i)}] Assume that $u$ is non-singular. Then, there exist $x_0\in\mathbb R^n$ and $\mu\in\mathbb R$ such that
			\begin{equation*}
				u(x)=\left(\frac{2\mu}{\mu^2+|x-x_0|^2}\right)^{\frac{4-n}{2}}.
			\end{equation*}
			\item[{\rm (ii)}] Assume that $u$ is singular. Then, there exist $a_0\in\mathbb (0,[n(n-4)/(n^2-4)]^{n-4/8})$ and $T\in[0,T_{a_0}]$ such that
			\begin{equation*}
				u(x)=|x|^{\frac{4-n}{2}}v_{a}(-\ln |x|+T),
			\end{equation*}
			where $v_{a}$ is the unique periodic bounded solution to the fourth order IVP below
			\begin{flalign}\tag{$\mathcal C_{4,\infty}$}\label{ourIVP4th}
				\begin{cases}
					v^{(4)}-\frac{n(n-4)+8}{2}v^{(2)}+\frac{n^2(n-4)^2}{16}v-\frac{n(n-4)^2(n^2-4)}{16}v^{\frac{n+4}{n-4}}=0 \quad {\rm in} \quad \mathbb{R}\\
					v(0)=a_0, \quad v^{(2)}(0)=a_2, \quad {\rm and} \quad v^{(1)}(0)=v^{(3)}(0)=0.
				\end{cases}
			\end{flalign}
		\end{itemize}
	\end{theoremletter}
	
	Second, we consider positive singular solutions $u\in C^{2}(\mathbb{R}^n\setminus\{0\})$ the second order critical equation below
	\begin{flalign}\tag{$\mathcal P_{2,\infty}$}\label{ourPDE2th}
		-\Delta u=\frac{n(n-4)}{8}u^{\frac{n+2}{n-2}} \quad {\rm in} \quad \mathbb{R}^n\setminus\{0\},
	\end{flalign}
	where $n\geqslant 3$, $\Delta$ is the Laplacian. 
	Let us notice that \eqref{ourPDE2th} is critical in the sense of the compact Sobolev embedding of $H^{1}(\mathbb R^n)$.
	All the aforementioned classification results were inspired by the classical theorems of R. H. Fowler \cite{fowler} (see also \cite{MR982351}) and T. Aubin \cite{MR0448404} and G. Talenti \cite{MR0463908} on the study of conformally equivalent metrics with constant scalar curvature, the so-called Yamabe problem.
	
	\begin{theoremletter}
		Let $u\in C^{2}(\mathbb{R}^n\setminus\{0\})$ be a positive  solution to \eqref{ourPDE2th}. Then, $u$ is radially symmetric with respect to the origin and monotonically decreasing. Moreover,
		\begin{itemize}
			\item[{\rm (i)}] if $u$ is non-singular, then there exist $x_0\in\mathbb R^n$ and $\mu\in\mathbb R$ such that
			\begin{equation*}
				u(x)=\left(\frac{2\mu}{\mu^2+|x-x_0|^2}\right)^{\frac{2-n}{2}},
			\end{equation*}
			which are called spherical solutions and they are denoted by $u_{x_0,\mu}$.
			\item[{\rm (ii)}] if $u$ is singular, then there exist $a_0\in\mathbb (0,[(n-2)/n]^{(n-2)/4})$ and $T\in[0,T_{a_0}]$ such that
			\begin{equation*}
				u(x)=|x|^{\frac{2-n}{2}}v_{a}(-\ln |x|+T),
			\end{equation*}
			where $v_{a}$ is the unique periodic bounded solution to the second order IVP below
			\begin{flalign}\tag{$\mathcal C_{2,\infty}$}\label{ourIVP2th}
				\begin{cases}
					v^{(2)}-\frac{(n-2)^2}{4}v+\frac{n(n-2)}{4}v^{\frac{n+2}{n-2}}=0 \quad {\rm in} \quad \mathbb{R}\\
					v(0)=a_0 \quad {\rm and} \quad v^{(1)}(0)=0.
				\end{cases}
			\end{flalign}
		\end{itemize}
	\end{theoremletter}

	Let us describe our strategy to prove Theorem~\ref{maintheorem}.
	First, we prove integrability and superharmonicity properties for solutions to \eqref{ourPDE}, which we use to perform an integral sliding technique and prove that positive singular solutions to \eqref{ourPDE} are radially symmetric.
	Second, we perform a change of variables to transform \eqref{ourPDE} into \eqref{ourIVP} a sixth order IVP with constant coefficients, which can be decomposed into three second order problems satisfying a maximum principle. 
	This allows us to define a Hamiltonian energy that is conserved along solutions, which, we use to prove some qualitative properties for positive solutions to \eqref{ourIVP}.
	Third, we use a two-parameter shooting technique to finish the proof of our classification result. 
	By undoing the Emden--Fowler change of variables, this can be easily translated into a result about solutions to \eqref{ourPDE}.
	
	The main difficulties in our approach are the lack of maximum principle for higher order operators and the possibly chaotic behavior enjoyed by solutions to ODEs driven by these operators. 
	The former can be overcome either by an integral representation formula or a decomposition into second order operators.
	To deal with the later issue, we use the conservation of the energy and qualitative properties previously proved. 
	These properties assure that this sixth order ODE behaves as a second order one.
	
	We remark that \eqref{ourPDE}, \eqref{ourPDE4th}, and \eqref{ourPDE2th} are particular cases of a more general class of equations, which we describe as follows.
	For any $m\in\mathbb{N}^+$, we are interested in classifying (smooth) positive singular solutions $u\in C^{2m}(\mathbb{R}^n\setminus\{0\})$ with $n>2m:=N$ (which will always be assumed so forth) to the following family of  critical even order poly-harmonic PDEs
	\begin{flalign}\tag{$\mathcal P_{N,\infty}$}\label{2mthPDE}
		(-\Delta)^mu=c_{n,N}u^{\frac{n+N}{n-N}} \quad {\rm in} \quad \mathbb{R}^n\setminus\{0\}.
	\end{flalign}
	Here $(-\Delta)^m$ is the $m$-poly-Laplacian and $c_{n,N}$ is a normalizing constant given by
	\begin{equation}\label{geoemtricconstanst}
		c_{n,N}=2^{N}{\Gamma\left(\frac{n+N}{4}\right)^2}{\Gamma\left(\frac{n-N}{4}\right)^{-2}},
	\end{equation}
	where $\Gamma(s)=\int_{0}^{\infty}\tau^{s-1}e^{-\tau}\ud\tau$ is the standard Gamma function.
	The power $\frac{n+N}{n-N}:=2_m^*-1$  is critical in the sense of the compact Sobolev embedding $H^{m}(\mathbb R^n)$.
	
	We remark that almost all of our technical results in this manuscript extend to this class of even-order poly-harmonic operators.
	In this regard, a classification result for entire singular solutions to \eqref{2mthPDE} in the sense of Theorem~\ref{maintheorem} is believed to be true. 
	Let us state this conjecture as follows
	\begin{conjecture}
		Let $u\in C^{N}(\mathbb{R}^n\setminus\{0\})$ be a positive singular solution to \eqref{ourPDE}. Then, $u$ is radially symmetric with respect to the origin and monotonically decreasing. Moreover, there exist $a_0\in\mathbb (0,a^{*}_{n,N})$ and $T\in[0,T_{a_0}]$ such that
		\begin{equation*}
			u(x)=|x|^{\frac{N-n}{2}}v_{a}(\ln |x|+T).
		\end{equation*}
		Here $v_{a}$ is the unique periodic bounded positive solution to the following higher order IVP
		\begin{flalign}\tag{$\mathcal C_{N,\infty}$}\label{2mthIVP}
			\begin{cases}
				\sum_{j=0}^{m}(-1)^{j+1}K_{2j,N}(n)v^{(2j)}=c_{n,N}v^{\frac{n+N}{n-N}} \quad {\rm in} \quad \mathbb{R}\\
				v^{(2j)}(0)=a_j \quad {\rm and} \quad v^{(2j-1)}(0)=0 \quad {\rm for} \quad j=1,\dots,m.
			\end{cases}
		\end{flalign}
		Here $K_{2j,N}(n),a^*_{n,N}>0$ are dimensional constants for all $j=1,\dots,m$, namely
		\begin{equation*}
			a^*_{n,N}=K_0^{\frac{n-N}{2N}} \quad {\rm and} \quad K_{2j,N}(n)=\sigma_j\left(\pm\sqrt{\lambda_{1,N}(n)},\dots,\pm\sqrt{\lambda_{m,N}(n)}\right),
		\end{equation*}
		where $\sigma_j$ are the $j$-th harmonic polynomial and
		\begin{equation*}
			\lambda_{j,N}(n)=\frac{n-N+4j}{2}
		\end{equation*}
		are the indicial roots of the constant coefficients ODE operator.
	\end{conjecture}
	
	We explain the plan for the rest of the manuscript.
	In Section~\ref{sec:preliminaries}, we introduce some definitions and prove some preliminary results that will be used subsequently, such as integral representation formulas, Kelvin transform, radial Emden--Fowler coordinates, decomposition result, and conservation of energy.
	In Section~\ref{sec:radialsymmetry}, we apply the integral moving spheres method to prove that solutions to \eqref{ourPDE} are radially symmetric.
	In Section~\ref{sec:ODEanalysis}, we perform some ODE analysis to prove  qualitative properties for solutions to \eqref{ourIVP}, which we use to perform a topological shooting
	method and classify these solutions.
	In Section~\ref{sec:proofofmaintheorem}, we combine the energy conservation, the radial symmetry, and the ODE analysis to prove Theorem~\ref{maintheorem}.
	
	\numberwithin{equation}{section} 
	\numberwithin{theorem}{section}
	
	\section{Preliminaries}\label{sec:preliminaries}
	This section aims to introduce some necessary background definitions and results for developing the sliding methods and the asymptotic analysis that will be later used in this manuscript.
	
	\subsection{Integral representation}\label{subsec:integral}
	Now we use a Green identity to transform the differential equation \eqref{ourPDE} into an integral equation.
	In this way, we can avoid using the classical form of the maximum principle, and a sliding method is available \cite{MR3558255, MR2055032}, which will be used to classify solutions. 	
	
	The next result uses the Green identity to convert \eqref{ourPDE} into an integral system. 
	
	\begin{proposition}\label{lm:integralrepresentation}
		Let $u\in C^{6}(\mathbb R^n)$ be a positive solution to \eqref{ourPDE}. 
		Then $($up to constant$)$, it follows
		\begin{equation}\label{globalintegralsystem}\tag{$\mathcal{I}_6$}
			u(x)=\int_{\mathbb{R}^n}|x-y|^{6-n}{f}(u) \ud y \quad {\rm in} \quad \mathbb{R}^n\setminus\{0\}.
		\end{equation}
	\end{proposition}
	
	\begin{proof}
		It directly follows from \cite[Theorem~4.3]{MR2200258} (see also \cite{MR2131045}).		
	\end{proof}
	
	\subsection{Kelvin transform}
	Later we will employ the moving spheres technique, which is based on the {$2m$-order Kelvin transform}. 
	For this, given $x_0\in\mathbb R^n$ and $\mu>0$, we need to establish the concept of {inversion about a sphere} $\partial B_{\mu}(x_0)$, which is given by $\mathcal{I}_{x_0,\mu}(x)=x_0+\mathcal K_{x_0,\mu}(x)^2(x-x_0)$, where $\mathcal K_{x_0,\mu}(x)=\mu/|x-x_0|$. 
	
	\begin{definition}
		For any $u\in C^{6}(\mathbb R^n\setminus\{0\})$, let us consider the sixth order Kelvin transform about the sphere with center at $x_0\in\mathbb{R}^n$ and radius $\mu>0$ defined by 
		\begin{equation*}
			u_{x_0,\mu}(x)=\mathcal K_{x_0,\mu}(x)^{n-6}u\left(\mathcal{I}_{x_0,\mu}(x)\right).
		\end{equation*}
	\end{definition}
	
	The next proposition states that solutions to \eqref{ourPDE} are invariant under the Kelvin transform.
	This holds because of the conformal invariance enjoyed by this family of critical equations.
	\begin{proposition}\label{prop:conformalinvariance}
		If $u$ is a solution to \eqref{ourPDE}, then  $u_{x_0,\mu}$ is a solution to 
		\begin{equation*}
			(-\Delta)^{3}u_{x_0,\mu}=f(u_{x_0,\mu}) \quad {\rm in} \quad \mathbb R^{n}\setminus\{0,x_0\}.
		\end{equation*}
	\end{proposition}
	
	\begin{proof}
		It directly follows by using \cite{MR1679783}.
	\end{proof}
	
	\subsection{Radial Emden--Fowler coordinates}\label{subsec:changeofvariables}
	This is section is devoted to constructing a change of variables that transforms the singular PDE \eqref{ourPDE} problem into the nice ODE problem with constant coefficients.
	Here we only consider functions that are radially symmetric, that is, $u=u(r)$ with $r=|x|$.
	Later on, we will prove that this in fact holds for positive singular solutions to \eqref{ourPDE} (see Proposition~\ref{prop:radialsymmetry}). 
	
	\begin{definition}
		For any $u\in C^{6}(\mathbb{R}^n\setminus\{0\})$ positive $($radial$)$ solution to \eqref{ourPDE}, we define the $($sixth order$)$ Emden--Fowler $($or logarithmic-cylindrical$)$ change of variables given by
		\begin{equation*}
			v(t)=|x|^{\frac{6-n}{2}}u(r) \quad {\rm with} \quad t=\ln r.
		\end{equation*}
		The power $\gamma_{n}:=\frac{6-n}{2}$ is chosen by conformal invariance.
	\end{definition}
	
	This change of variables is used to transform the singular PDE \eqref{ourPDE} problem into the nice ODE problem with constant coefficients 
	\begin{equation}\tag{$\mathcal{O}_{6}$}\label{ourODE}
		v^{(6)}-K_4v^{(4)}+K_2v^{(2)}+g(v)=0 \quad {\rm in} \quad \mathbb R.
	\end{equation}
	The nonlinear term is $g: C^{6}(\mathbb R)\rightarrow C^6(\mathbb R)$ is defined as
	\begin{equation*}
		g(v):=f(v)-K_0v \quad {\rm where} \quad f(v):={c}_n|v|^{\frac{12}{n-6}}v.
	\end{equation*}
	The operator $P^3_{\rm rad}:C^{6}(\mathbb R)\rightarrow C(\mathbb R)$ is the so-called (sixth order) logarithmic cylindrical Paneitz operator (restricted to radial functions), and is given by
	\begin{equation}\label{cylindricaloperator}
		P^3_{\rm rad}=\partial_t^{(6)}-K_4\partial_t^{(4)}+K_2\partial_t^{(2)}-K_0,
	\end{equation}
	where
	\begin{align}\label{coefficients}\nonumber
		K_0&=\frac{1}{64}(n-6)^2(n-2)^2(n+2)^2,\\ K_2&=\frac{1}{16}(3n^4-24n^3+72n^2-96n+304),\\\nonumber 
		K_4&=\frac{1}{4}(3n^2-12n+44).
	\end{align}
	
	\subsection{Decomposition}\label{subsec:decomposition}
	We prove a decomposition result for the tri-Laplacian operator written in Emden--Fowler coordinates.
	Namely, we decompose $P^3_{\rm rad}$ into a composition of three second order operators satisfying a comparison principle.
	In what follows, given $\lambda>0$, we denote $L_{\lambda}:=\partial_t^{(2)}-\lambda$.
	
	We start with the definition of indicial roots for a linear operator.
	\begin{definition}
		Let ${L}\in (C^{6}(\mathbb R))^{\prime}$ be a linear operator.
		The indicial roots of $L$ at $+\infty$ $($resp. $-\infty$$)$ are $\lambda\in\mathbb R$ for which there is a non-zero function $v \in {C}^{6}(\mathbb R)$ and $\lambda^{\prime}<\lambda$ $($resp. $\lambda^{\prime}>\lambda$$)$ such that $e^{-\lambda^{\prime} t} L(e^{\lambda t}v(t))\rightarrow0$ as $t\rightarrow+\infty$ $($resp. $t\rightarrow-\infty$$)$. We denote by $\mathcal{I}(L)\subset\mathbb R$ its set of indicial roots.
	\end{definition}
	
	\begin{remark}
		Notice that when ${L}\in (C^{6}(\mathbb R))^{\prime}$ has constant coefficients, that is, $L=\sum_{j=0}^{6}k_j\partial_t^{(j)}$, it follows that $\mathcal{I}(L)=\{p_{L}(0)\}^{-1}$, where $p_{L}(\lambda)=\sum_{j=0}^{6}k_j\lambda^{j}$ is the indicial polynomial.
		In particular, a direct computation shows that the indicial equation associated to \eqref{cylindricaloperator} are given by
		\begin{equation}\label{indicialequation}
			p_{\lambda}(P^3_{\rm rad})=\lambda^6-K_4\lambda^4+K_2\lambda^2-K_0=0,
		\end{equation}
		which have a strictly positive discriminant, that is,
		\begin{equation}\label{discriminant}
			{\rm disc}_{\lambda}(p):=-27K_0^2-4K_2^2-4K_4^2+K_4^2K_2^2+18K_4K_2K_0>0.
		\end{equation}
		This in turns guarantees the existence of distinct real roots $\pm\sqrt{\lambda_1},\pm\sqrt{\lambda_2},\pm\sqrt{\lambda_3}\in\mathbb R$ solving \eqref{indicialequation}.
	\end{remark}
	
	Next, we prove a decomposition of $P^3_{\rm rad}$ into three second order operators.
	
	\begin{proposition}\label{prop:decomposition}
		The following decomposition holds
		\begin{equation}\label{decomposition}
			P^3_{\rm rad}=L_{\lambda_1}\circ L_{\lambda_2}\circ L_{\lambda_3},
		\end{equation}
		where
		\begin{equation}\label{indicialroots}
			\lambda_1=\frac{n-6}{2}, \quad  \lambda_2=\frac{n-2}{2}, \quad {\rm and} \quad \lambda_3=\frac{n+2}{2}.
		\end{equation}
	\end{proposition}
	
	\begin{proof}
		Initially, recall that for $q>6-n$, the following formula holds in the sense of distributions,
		\begin{equation*}
			(-\Delta)^3(|x|^{q})=\prod_{j=0}^{3}(q-2j)\prod_{j=1}^{3}(q-2j+n)|x|^{q-6} \quad {\rm in} \quad \mathbb R^n\setminus\{0\},
		\end{equation*}
		which implies that $\{u_j\}_{j\in\{1,\dots,6\}}\subset C^{6}(\mathbb R^n\setminus\{0\})$ is a basis for the space of radial solutions to the homogeneous tri-harmonic equation, where $u_j(|x|)=|x|^{q_j}$ with $q_j=2(j-1)$ if $j=1,2,3$ and $2(j-m)-n$ if $j=4,3,5$.
		In other terms, $(-\Delta)^3 u_j=0$ in $\mathbb R^n\setminus\{0\}$ for any $j=1,2,3,4,5,6$.
		Hence, under the Emden--Fowler change of variables, we have that $\{v^\pm_1,v^\pm_2,v^\pm_3\}\subset C^{6}(\mathbb R)$, where $v^{\pm}_j(t)=e^{\pm\sqrt{\lambda_j} t}$ with $\lambda_j=q_j-\gamma_{n}$ for $j=1,2,3$, forms a basis for the space of solutions to \eqref{ourODE}, that is, $\{v^\pm_1,v^\pm_2,v^\pm_3\}\subset C^{6}(\mathbb R)$ is linearly independent and $P^3_{\rm rad} v^{\pm}_1=P^3_{\rm rad} v^{\pm}_2=P^3_{\rm rad} v^{\pm}_3=0$ in $\mathbb R$.
		Then, the set of indicial roots of $\mathcal{I}(P^3_{\rm rad})=\{\pm\sqrt{\lambda_1},\pm\sqrt{\lambda_2},\pm\sqrt{\lambda_3}\}\subset \mathbb R$.
		From this, it is straightforward to check that \eqref{decomposition} holds.
	\end{proof}

	\subsection{Conservation of energy}\label{subsec:conservation}
	In this section, we find a quantity that is conserved along solutions to \eqref{ourODE}.
	These will be called the Hamiltonian energy associated with the sixth order ODE \eqref{ourODE}.
	
	Let us start with the classical definition of Hamiltonian energy for a solution $v\in C^6(\mathbb R)$ to \eqref{ourODE}, which is obtained by multiplying \eqref{ourODE} by $v^{(1)}$ and integrating by parts. 
	\begin{definition}
		For any real function $v\in C^{6}(\mathbb R)$, let us define its Hamiltonian energy $($with respect to \eqref{ourODE}$)$ $\mathcal{H}:\mathbb R\times C^{6}(\mathbb R)\rightarrow\mathbb R$ by
		\begin{equation}\label{hamiltonianenergy}
			\mathcal{H}(t,v):=\left(v^{(5)}v^{(1)}-v^{(4)}v^{(2)}+\frac{1}{2}{v^{(3)}}^2\right)-{K_4}\left(v^{(3)}v^{(1)}-\frac{1}{2}{v^{(2)}}^{2}\right)+\frac{K_2}{2}{v^{(1)}}^2-\frac{K_0}{2}v^2+F(v),
		\end{equation}
		where
		\begin{equation*}
			G(v):=F(v)-\frac{K_0}{2}v^2
			\quad {\rm and} \quad F(v):=\hat{c}_n|v|^{\frac{2n}{n-6}}
		\end{equation*}
		and
		\begin{equation*}
			\hat{c}_{n}:=\frac{(n-6)^2(n^4-20n^2+64)}{128}.
		\end{equation*}
	\end{definition}
	
	\begin{remark}
		Notice that the real roots of $G$ are the equilibrium solutions to \eqref{ourODE}, namely
		\begin{equation}
			v\equiv0 \quad {\rm and} \quad v\equiv\pm a^{*}_n,
		\end{equation}
		where
		\begin{equation}\label{constantsolutions}
			a^{*}_n:=K_0^{\frac{n-6}{12}}=\left(\frac{(n-6)(n-2)(n+2)}{8}\right)^{\frac{n-6}{6}}.
		\end{equation}
	\end{remark}
	
	Next, it is direct to prove that this energy is conserved along with solutions to \eqref{ourODE}. 
	We emphasize that this conservation is a local property and valid on the maximal interval of existence and does not require any {\it a priori} boundedness assumption.
	\begin{proposition}\label{prop:conservation}
		If $v\in C^{6}(\mathbb R)$ is a positive solution to \eqref{ourODE}, then 
		\begin{equation*}
			\dfrac{\partial}{\partial t}\mathcal{H}(t,v)\equiv0.
		\end{equation*}
		In other terms, there exists $\mathcal H_{v}\in\mathbb R$ such that
		\begin{equation}\label{conservationofenergy}
			\mathcal{H}(v)(t)=\mathcal{H}(t,v)\equiv\mathcal{H}(v):=\mathcal{H}_{v}.
		\end{equation}
	\end{proposition}
	
	\begin{proof}
		It is a direct computation.
	\end{proof}
	
	\begin{remark}
		It is convenient to write the Hamiltonian energy as
		\begin{align}\label{energyfactorized}
			\mathcal{H}(v)=\frac{1}{2}{v^{(3)}}^2+\mathcal{E}_2(v)v^{(2)}+\mathcal{E}_1(v)v^{(1)}+G(v),         
		\end{align}
		where                 \begin{align}\label{auxiliarysignfunctions}
			\mathcal{E}_1(v):=v^{(5)}-{K_{4}}v^{(3)}+\frac{K_{2}}{2}v^{(1)} \quad {\rm and} \quad  \mathcal{E}_2(v):=-v^{(4)}+\frac{K_{4}}{2}v^{(2)}
		\end{align}
		are the so-called auxiliary energy summand functions.
	\end{remark}
	
	\section{Radial symmetry}\label{sec:radialsymmetry}
	In this section, using the integral moving spheres technique, we prove that singular solutions to \eqref{ourPDE} are radially symmetric.
	The main ingredient in the proof is the integral version of the moving spheres technique contained in \cite{MR2479025,MR3694645,MR2055032,arxiv:1901.01678}.
	For this, we are based on the conformal invariance enjoyed by solutions to \eqref{ourPDE} in Proposition~\ref{prop:conformalinvariance}.
	
	\begin{proposition}\label{prop:radialsymmetry}
		Let $u\in C^{6}(\mathbb R^n\setminus\{0\})$ be a positive singular solution to equation \eqref{ourPDE}. Then, $u$ is radially symmetric about the origin and monotonically decreasing.
	\end{proposition}
	
	Before we give the proof of the rotational symmetry for singular solutions to \eqref{ourPDE}, we need to establish a set of preliminary results.
	
	\subsection{Integrability}\label{subsec:integrability}
	We show that any positive singular solution to \eqref{ourPDE} is distributional.
	
	\begin{lemma}\label{lm:integrability}
		Let $u\in C^{6}(\mathbb R^n\setminus\{0\})$ be a positive singular solution to \eqref{ourPDE}. Then, it holds $u \in L^{\frac{n+6}{n-6}}(\mathbb{R}^n)$. 
		In particular, $u$ is a
		distribution solution to \eqref{ourPDE}.
	\end{lemma}
	
	\begin{proof}
		For any $0<\varepsilon \ll 1$, let us consider $\eta_{\varepsilon}\in C^{\infty}(\mathbb{R}^{n})$ with $0\leqslant\eta_{\varepsilon}\leqslant1$ satisfying
		\begin{equation}\label{cutoff}
			\eta_{\varepsilon}(x)=
			\begin{cases}
				0, & \mbox{if} \ |x| \leqslant \varepsilon\\
				1, & \mbox{if} \ |x| \geqslant 2 \varepsilon,
			\end{cases}
		\end{equation}
		and $|D^{(j)} \eta_{\varepsilon}| \leqslant C_0 \varepsilon^{-j}$ in $\mathbb{R}^{n}$ for $j=1,2,3,4,5,6$ and some $C_0>0$.
		Define $\xi_{\varepsilon}=\left(\eta_{\varepsilon}\right)^{\frac{n+6}{2}}$. Multiplying \eqref{ourPDE} by $\xi_{\varepsilon}$, and integrating by parts in $B_{r}$ with $r\in(1/2,1)$, we obtain
		\begin{equation*}
			\int_{\mathbb{R}^{n}}f(u) \xi_{\varepsilon}\ud x =\int_{\mathbb{R}^{n}} u(-\Delta)^{3} \xi_{\varepsilon}\ud x.
		\end{equation*}
		One can verify that there exists $C_1>0$ such that
		\begin{equation*}
			\left|(-\Delta)^{3} \xi_{\varepsilon}\right| \leqslant C_1 \varepsilon^{-6}\eta_{\varepsilon}^{\frac{n-6}{2}} \chi_{\{\varepsilon \leqslant|x| \leqslant 2 \varepsilon\}}=C_1 \varepsilon^{-6}\xi_{\varepsilon}^{\frac{n-6}{n+6}} \chi_{\{\varepsilon \leqslant|x| \leqslant 2 \varepsilon\}},
		\end{equation*}
		which, by H\"{o}lder's inequality, gives us
		\begin{align*}
			\left|\int_{\mathbb{R}^{n}} u(-\Delta)^{3} \xi_{\varepsilon}\ud x\right|
			&\leqslant C_1 \varepsilon^{-6} \int_{\{\varepsilon \leqslant|x| \leqslant 2 \varepsilon\}} u\xi_{\varepsilon}^{\frac{n-6}{n+6}}\ud x \\
			& \leqslant \tilde C_1 \varepsilon^{-6} \varepsilon^{\frac{12n}{n+6}}\left(\int_{\{\varepsilon \leqslant|x| \leqslant 2 \varepsilon\}} f(u) \xi_{\varepsilon}\ud x\right)^{\frac{n-6}{n+6}} \\
			& \leqslant \hat C_1\left(\int_{\{\varepsilon \leqslant|x| \leqslant 2 \varepsilon\}} f(u) \xi_{\varepsilon}\ud x\right)^{\frac{n-6}{n+6}}.
		\end{align*}
		Thus, it follows
		\begin{equation*}
			\int_{\mathbb{R}^{n}} f(u) \xi_{\varepsilon}\ud x \leqslant \hat C_1\left(\int_{\{\varepsilon \leqslant|x| \leqslant 2 \varepsilon\}} f(u) \xi_{\varepsilon}\ud x\right)^{\frac{n-6}{n+6}},
		\end{equation*}
		from which one can find a constant $C_2>0$ (independent of $\varepsilon$) such that
		\begin{equation*}
			\int_{\mathbb{R}^{n}}f(u) \xi_{\varepsilon}\ud x \leqslant C_2.
		\end{equation*}
		Now letting $\varepsilon \rightarrow 0$, we conclude that $u\in L^{\frac{n+6}{n-6}}(\mathbb{R}^{n})$ and the integrability follows.
		
		For any nonnegative $\phi\in C^{\infty}_c(\mathbb{R}^n)$, we multiply \eqref{ourPDE} by $\widetilde{\phi}=\eta_{\varepsilon} \phi$, where $\eta_{\varepsilon}$ is given by \eqref{cutoff}. Then, using that $u\in L^{\frac{n+6}{n-6}}(\mathbb{R}^{n})$ and integrating by parts, we get
		\begin{equation}\label{yang1}
			\int_{\mathbb{R}^n} u(-\Delta)^{3}\left(\eta_{\varepsilon}\phi\right)\ud x=\int_{\mathbb{R}^n} f(u)\eta_{\varepsilon}\phi\ud x.
		\end{equation}
		By a direct computation, we find that $\Delta^{3}(\eta_{\varepsilon} \phi)=\eta_{\varepsilon} \Delta^{3}\phi+\psi_{\varepsilon}$, where
		\begin{equation*}
			\psi_{\varepsilon}=6 \nabla \eta_{\varepsilon}\nabla \Delta^2\phi-15 \Delta \eta_{\varepsilon} \Delta^2\phi+20 \nabla \Delta\eta_{\varepsilon}\nabla \Delta\phi-15 \Delta^2\eta_{\varepsilon}\Delta\phi+6 \nabla \Delta^2\eta_{\varepsilon}\nabla\phi-\phi\Delta^{3} \eta_{\varepsilon}.
		\end{equation*}
		Furthermore, using H\"{o}lder's inequality again, one can find $C_3>0$ such that
		\begin{align*}
			\left|\int_{\mathbb{R}^n} u\psi_{\varepsilon}\ud x\right| \leqslant C_3\left(\int_{\{\varepsilon\leqslant|x| \leqslant 2 \varepsilon\}} f(u)\ud x\right)^{\frac{n-6}{n+6}} \rightarrow 0 \quad \mbox{as} \quad \varepsilon \rightarrow 0.
		\end{align*}
		Finally, letting $\varepsilon \rightarrow 0$ in \eqref{yang1}, and applying the dominated convergence theorem the proof follows.
	\end{proof}
	
	\subsection{Superharmonicity}\label{subsec:superharmonicity}
	We show that positive singular solutions to \eqref{ourPDE} satisfy a superharmonicity property for lower orders powers of the Laplacian.
	
	\begin{lemma}\label{lm:superharmonicity1}
		Let $u\in C^{6}(\mathbb R^n\setminus\{0\})$ be a positive singular solution to \eqref{ourPDE}. Then, $u$ is a superharmonic in the distributional sense, that is,  $\phi\in C^{\infty}_c(\mathbb{R}^{n})$, one has
		\begin{equation*}
			\int_{\mathbb{R}^{n}}\Delta u \Delta^{2}\phi\ud x\leqslant 0 \quad {\rm for \ all} \quad \phi\in C^{\infty}_c(\mathbb{R}^{n}).
		\end{equation*}
		Moreover, in the classical sense
		\begin{equation}\label{superhmarnocity1}
			-\Delta u\geqslant0 \quad {\rm in} \quad \mathbb{R}^{n} \setminus\{0\}.
		\end{equation}
	\end{lemma}
	
	\begin{proof}
		Proceeding similarly to Lemma~\ref{lm:integrability}, one can prove that $u \in L_{\rm loc}^{\frac{n+6}{n-6}}(\mathbb{R}^{n})$. 
		Let $\eta_{\varepsilon} \in C^{\infty}(\mathbb{R}^{n})$ be the cut-off function given by \eqref{cutoff} and $\phi \in C_{c}^{\infty}(\mathbb{R}^{n})$ be a nonnegative test function. 
		Then, by multiplying \eqref{ourPDE} by $\eta_{\varepsilon} \phi$, and integrating by parts, we get
		\begin{align*}
			0 \geqslant -\int_{\mathbb{R}^{n}} \eta_{\varepsilon} \phi f(u)\ud x=\int_{\mathbb{R}^{n}} \Delta^2\left(\eta_{\varepsilon}\phi\right) \Delta u\ud x=\int_{\mathbb{R}^{n}} \Delta u\left(\Delta^2 \phi \eta_{\varepsilon}+\varsigma^{\varepsilon}\right)\ud x,
		\end{align*}
		where $\varsigma^{\varepsilon}:=4 \langle\nabla \Delta \phi, \nabla \eta_{\varepsilon}\rangle+6 \Delta \phi \Delta \eta_{\varepsilon}+4 \langle\nabla \phi, \nabla \Delta \eta_{\varepsilon}\rangle+\phi\Delta^2\eta_{\varepsilon}$.
		Notice that  $\varsigma^{\varepsilon}(x) \equiv 0$ when $|x| \leqslant \varepsilon$ or $|x| \geqslant 2 \varepsilon$, and $|\Delta \varsigma^{\varepsilon}(x)| \leqslant C_1 \varepsilon^{-6}$, for some $C_1>0$. 
		In addition, since $n-6-\frac{n(n-6)}{n+6}>0$, the following estimate holds
		\begin{align*}
			\left|\int_{\mathbb{R}^{n}} \Delta u \varsigma^{\varepsilon}\ud x\right| & \leqslant \int_{\mathbb{R}^{n}} u|\Delta \varsigma^{\varepsilon}|\ud x\\
			& \leqslant C_1 \varepsilon^{-6}\left(\int_{\{\varepsilon \leqslant|x| \leqslant 2 \varepsilon\}} f(u)\ud x\right)^{\frac{n-6}{n+6}}\varepsilon^{n\left(1-\frac{n-6}{n+6}\right)} \\
			& \leqslant C_1 \varepsilon^{n-6-\frac{n(n-6)}{n+6}} \rightarrow 0 \quad \mbox{as} \quad \varepsilon \rightarrow 0,
		\end{align*}
		which implies
		\begin{align*}
			\int_{\mathbb{R}^{n}} \Delta u\Delta \phi\ud x =\lim_{\varepsilon \rightarrow 0}\int_{\mathbb{R}^{n}}\left(\Delta u\Delta^2 \varsigma^{\varepsilon}\ud x+\eta_\varepsilon\Delta\phi \Delta u\right)\ud x=\int_{\mathbb{R}^{n}}\phi f(u)\ud x\leqslant 0.
		\end{align*}
		Thus, $-\Delta u$ is superharmonic in the whole space $\mathbb{R}^{n}$ in the distributional sense, which gives the first part of the proof.
		
		To prove the second statement, given $0<\varepsilon\ll1$, let us consider $\widetilde{u}^{\varepsilon}:=-\Delta u+\varepsilon$.
		Using Lemma~\ref{lm:integrability}, there exists a constant $C_2>0$, depending only on $n$ and $s$, such that for all $|x| \geqslant 4$, it holds
		\begin{equation*}
			\sum_{j=0}^5|x|^{\gamma_n+j}|D^{(j)} u(x)| \leqslant C_2,
		\end{equation*}
		which yields that  $\lim_{|x|\rightarrow \infty}|\Delta u(x)|=0$.
		Whence, for any $0<\varepsilon\ll1$, there exists $R_{\varepsilon}\gg1$ such that $\widetilde{u}^\varepsilon>{\varepsilon}/{2}$ for $|x| \geqslant R_{\varepsilon}$.
		Finally, using that $\widetilde{u}^{\varepsilon}$ is superharmonic in $\mathbb{R}^{n}$ in the distributional sense, we have that
		$\widetilde{u}^{\varepsilon} \geqslant 0$ in $\mathbb{R}^{n}\setminus\{0\}$, which, by passing to the limit as $\varepsilon\rightarrow 0$, provides $-\Delta u \geqslant 0$ in $\mathbb{R}^{n}\setminus\{0\}$.
		The last inequality concludes the proof of the lemma.
	\end{proof}
	
	\begin{lemma}\label{lm:superharmonicity2}
		Let $u\in C^{6}(\mathbb R^n\setminus\{0\})$ be a positive singular solution to \eqref{ourPDE}. Then, $-\Delta u$ is a superharmonic in the distributional sense, that is,  $\phi\in C^{\infty}_c(\mathbb{R}^{n})$, one has
		\begin{equation*}
			\int_{\mathbb{R}^{n}}\Delta^2 u \Delta\phi\ud x\geqslant 0 \quad {\rm for \ all} \quad \phi\in C^{\infty}_c(\mathbb{R}^{n}).
		\end{equation*}
		Moreover, in the classical sense
		\begin{equation}\label{superhmarnocity2}
			\Delta^2 u\geqslant0 \quad {\rm in} \quad \mathbb{R}^{n} \setminus\{0\}.
		\end{equation}
	\end{lemma}
	
	\begin{proof}
		The proof is similar to the one of Lemma~\ref{lm:superharmonicity1}, so we omit it.
	\end{proof}
	
	\begin{remark}
		Notice that these superharmonicity identities in \eqref{superhmarnocity1} and \eqref{superhmarnocity1} can be respectively rewritten in radial Emden--Fowler coordinates $($see Subsection~\ref{subsec:changeofvariables}$)$ as follows
		\begin{equation*}
			v^{(2)}+(n-5)v^{(1)}+\frac{(n-6)(n-4)}{4}v<0
		\end{equation*}
		and
		\begin{equation*}
			v^{(4)}+2(n-3)v^{(3)}-{(3n^2-18n+22)}v^{(2)}+{(n^3-9n^2+22n-12)}v^{(1)}+\frac{(n-6)(n-4)(n-2)}{16}v>0.
		\end{equation*}
		In addition, we have the decomposition
		\begin{equation*}
			v^{(4)}-{(\nu^+_n+\nu^-_n)}v^{(2)}+(\nu^+_n\nu^-_n) v>0,
		\end{equation*}
		where
		\begin{equation*}
			\nu^\pm:=\sqrt{5\pm\sqrt{2n^4-24n^3+88n^2-96n+16}}.
		\end{equation*}
	\end{remark}
	
	\subsection{Integral moving spheres method}\label{subsec:movingspheres}
	We prove the main result of this subsection, namely, the radial symmetry of the blow-up limit solutions. 
	
	\begin{lemma}\label{msestimate}
		Let $u\in C^{6}(\mathbb R^n\setminus\{0\})$ be a positive singular solution to \eqref{ourPDE}. 
		For any $x \in \mathbb{R}^n$, $z \in \mathbb R^n\setminus\left(\{0\}\cup B_{\mu}(x)\right)$ and $\mu\in(0,1)$, it holds that $u(z)-(u)_{x, \mu}(z)>0$.
	\end{lemma}
	
	\begin{proof}
		Let $u$ be a positive singular solution to \eqref{ourPDE}. 
		Using the identities in \cite[page 162]{MR2055032}, one has 
		\begin{equation*}
			\left(\frac{\mu}{|z-x|}\right)^{n-6}\int_{|y-x| \geqslant\mu} {\left|\mathcal{I}_{x,\mu}(z)-y\right|^{n-6}}{f(u(y))} \ud y=\int_{|y-x| \leqslant \mu} \left|z-y\right|^{n-6}{f(u(y))}\ud y
		\end{equation*}
		and
		\begin{equation*}
			\left(\frac{\mu}{|z-x|}\right)^{n-6} \int_{|y-x|\leqslant\mu} {\left|\mathcal{I}_{x,\mu}(z)-y\right|^{n-6}}{f(u(y))} \ud y=\int_{|y-x|\geqslant \mu} \left|z-y\right|^{n-6}{f(u(y))}\ud y,
		\end{equation*}
		which yields
		\begin{equation}\label{integralconformalinvariance}
			{(u)}_{x, \mu}(z)=\int_{\mathbb{R}^{n}} \left|z-y\right|^{n-6}{f(u(y))}\ud y \quad \mbox{for} \quad z\in \mathcal{I}_{x,\mu}(\mathbb R^n).
		\end{equation}
		Consequently, for any $x \in \mathbb{R}^n$ and $\mu<1$, we have that for $z\in \mathbb{R}^n\setminus\{0\}\cup B_{\mu}(x)$,
		\begin{equation*}
			u(z)-(u)_{x, \mu}(z)=\int_{|y-x|\geqslant\mu} E(x,y,\mu,z)\left[f(u(y))-f(u_{x,\mu}(y))\right]\ud y,
		\end{equation*}
		where 
		\begin{equation}\label{kernelkelvintransform}
			E(x,y,z,\mu):={|z-y|^{6-n}}-\left(\frac{|z-x|}{\mu}\right)^{6-n} {\left|\mathcal{I}_{x,\mu}(z)-y\right|^{6-n}}
		\end{equation}
		is used to estimate the difference between $u$ and its Kelvin transform $u_{x,\mu}$. Finally, using its decay properties, it is straightforward to check that $E(x,y,z,\mu)>0$ for all $|z-x|>\mu>0$, which concludes the proof.
	\end{proof}
	
	Next, let us introduce the {critical sliding parameter} as the supremum for which an inequality relating a component function and its Kelvin transform is satisfied.
	\begin{definition}
		Given $x \in \mathbb{R}^{n}$, let us define
		\begin{equation}\label{criticalparameter}
			\mu_*(x)=\sup \left\{\mu>0:(u)_{x,r}\leqslant u \ {\rm in} \ \mathbb{R}^{n}\setminus B_{r}(x) \ {\rm for \ any} \ 0<r<\mu\right\}.
		\end{equation}
		Since $u$ is such that $-\Delta u>0$ and $\Delta^{2}u>0$, we get $\mu^*(x)>0$.
	\end{definition}
	
	The next lemma is essentially the moving spheres technique in its integral form.
	This method provides the exact form for any blow-up limit solution to \eqref{ourPDE}, which depends on whether the critical sliding parameter $\mu^*(x)$ is finite or infinite. 
	
	\begin{lemma}\label{lm:movingspheres}
		Let $u\in C^{6}(\mathbb R^n\setminus\{0\})$ be a positive singular solution to \eqref{ourPDE}, $z\in\mathbb{R}^n$ and $\mu^*(z)>0$ given by \eqref{criticalparameter}. 
		The following holds:
		\begin{itemize}
			\item[{\rm (i)}] if $\mu^*(x)<\infty$ is finite, then 
			$u_{x,\mu^*(x)}=u$ in $\mathbb{R}^{n} \setminus\{0,x\}$.
			\item[{\rm (ii)}] if $\mu^*(x_0)=\infty$, for some $x_{0} \in \mathbb{R}^{n}$, then $\mu^*(x)=\infty$ for all $x \in \mathbb{R}^{n}\setminus\{0\}$.
		\end{itemize}
	\end{lemma}
	
	\begin{proof}
		Without loss of generality, we may assume $x_0=0$. 
		Let us fix $\mu^*=\mu^*(0)$ and $(u)_{\mu}=(u)_{0,\mu}$. 
		By the definition of $\mu^*(x)>0$, we have $(u)_{\mu^*}(x) \leqslant u(x)$ for all $|x| \geqslant \mu^*$.
		Thus, by \eqref{integralconformalinvariance}, with $x=0$ and $\mu=|x|\geqslant \mu^*$, and the positivity of the kernel $E(0,y,z,\mu)$ given by \eqref{kernelkelvintransform}, either $u_{\mu^*}(y)=u(y)$ for all $|x|\geqslant \mu^*$ or $(u)_{\mu^*}(y)<(u)(y)$ for all $|x|>\mu^*$. In the former case, the conclusion easily follows.
		In the sequel, we assume that the last condition holds.
		Hence, the integral representation in Proposition~\ref{lm:integralrepresentation} yields
		\begin{align*}
			\liminf_{|z|\rightarrow\infty}|z|^{n-6}\left[u(z)-(u)_{\mu^*}(z)\right]&=\liminf_{|z|\rightarrow \infty} \int_{|y| \geqslant \mu^*}|z|^{n-6} E(0,y,z,\mu^*)\left[{f}(u(y))-{f}(u_{\mu^*}(y))\right]\ud y&\\
			&\geqslant \int_{|y|\geqslant \mu^*}\left(1-\left(\frac{\mu^*}{|y|}\right)^{n-6}\right)\left[{f}(u(y))-{f}(u_{\mu^*}(y))\right]\ud y>0,&
		\end{align*}
		which implies that there exists $\varepsilon_{1} \in(0,1)$ satisfying $u(z)-(u)_{\mu^*}(z)\geqslant {\varepsilon_{1}}{|z|^{6-n}}$ for all $|z| \geqslant \mu^*+1$.
		Moreover, there exists $\varepsilon_{2}\in(0,\varepsilon_{1})$ such that,
		for $|z|\geqslant \mu^*+1 \ \mbox{and} \ \mu^* \leqslant \mu \leqslant \mu^*+\varepsilon_{2}$, we find 
		\begin{equation}\label{blowup1}
			\left(u-(u)_{\mu^*}\right)(z) \geqslant {\varepsilon_{1}}{|z|^{6-n}}+\left((u)_{\mu^*}-(u)_{\mu}\right)(z) \geqslant \frac{\varepsilon_{1}}{2}|z|^{6-n}.
		\end{equation}
		Whence, for any $\varepsilon \in\left(0, \varepsilon_{2}\right)$ (to be chosen later), $\mu^* \leqslant \mu \leqslant \mu^*+\varepsilon$, and $\mu \leqslant|y| \leqslant \mu^*+1$, we have
		\begin{align*}
			\left(u-(u)_{\mu^*}\right)(z)&=\int_{|y| \geqslant \mu} E(0,y,z,\mu)\left[{f}(u(y))-{f}(u_{\mu}(y))\right]\ud y\\
			&\geqslant \int_{\mu^* \leqslant|y| \leqslant \mu^*+1} E(0,y,z,\mu)\left[{f}(u_{\mu^*}(y))-{f}(u_{\mu}(y))\right]\ud y\\
			&+\int_{\mu^*+2 \leqslant|y| \leqslant \mu^*+3} E(0,y,z,\mu)\left[{f}(u(y))-{f}(u_{\mu}(y))\right] \ud y.
		\end{align*}
		Now using \eqref{blowup1}, there exists $\delta_{1}>0$ such that ${f}(u(y))-{f}(u_{\mu}(y))\geqslant\delta_1$ for $\mu^*+2 \leqslant|y| \leqslant \mu^*+3$.
		Since $E(0, y, z, \mu)=0$ for all $|z|=\mu$ and
		\begin{equation*}
			\nabla_{z} E(0, y, z, \mu)\cdot z\big|_{|z|=\mu}=(n-4)|z-y|^{8-n}\left(|z|^{2}-|y|^{2}\right)>0 \quad \mbox{for all} \quad \mu^*+2 \leqslant|y| \leqslant \mu^*+3,
		\end{equation*}
		where $\delta_{2}>0$ is a constant independent of $\varepsilon$. 
		Then, there exists $C_1>0$ such that, for $\mu^* \leqslant \mu \leqslant \mu^*+\varepsilon$, we get
		\begin{equation*}
			\left|{f}(u(y))-{f}(u_{\mu^*}(y))\right| \leqslant C_1(\mu-\mu^*) \leqslant C_1 \varepsilon \quad \mbox{for all} \quad \mu^* \leqslant \mu \leqslant|y| \leqslant \mu^*+1.
		\end{equation*}
		Furthermore, recalling that $\mu \leqslant|z| \leqslant \mu^*+1$, one can find $C_2>0$ such that
		\begin{align*}
			\int_{\mu \leqslant|y| \leqslant \mu^*} E(0, y, z, \mu) \ud y &\leqslant \left|\int_{\mu \leqslant|y| \leqslant \mu^*+1}\left[{|y-z|^{6-n}}-{\left|\mathcal{I}_{\mu}(z)-y\right|^{8-n}}+\left(\frac{\mu}{|z|}-1\right)^{n-6}\left|\mathcal{I}_{\mu}(z)-y\right|^{n-4}\right] \ud y\right|\\
			&\leqslant C_2\left|\mathcal{I}_{\mu}(z)-z\right|+C_2(|z|-\mu) \leqslant C_2(|z|-\mu),
		\end{align*}
		which, provides that for small $0<\varepsilon\ll1$, $\mu^* \leqslant \mu \leqslant \mu^*+\varepsilon$, and $\mu\leqslant|z| \leqslant \mu^*+1$, it follows
		\begin{align*}
			\left(u-(u)_{\mu^*}\right)(z)& 
			\geqslant-C_2 \varepsilon \int_{\mu \leqslant|y| \leqslant \mu^*+1} E(0, y, z, \mu)\ud y+\delta_{1}\delta_{2}(|z|-\mu) \int_{\mu^*+2 \leqslant|z| \leqslant \mu^*+3} \ud y&\\
			&\geqslant\left(\delta_{1}\delta_{2} \int_{\mu^*+2 \leqslant|y| \leqslant \mu^*+3}\ud z-C_2\varepsilon\right)(|z|-\mu)&\\
			&\geqslant 0.&
		\end{align*}
		This is a contradiction to the definition of $\mu^*(x)>0$. Therefore, the first part of the lemma is established.
		
		Next, by the definition of $\mu^*(x)>0$, we know that $u_{x, \mu}(z)\leqslant u(z)$ for all $0<\mu<\mu^*(x)$, with $|z-x|\geqslant\mu$; thus, multiplying it by $|z|^{n-6}$, and taking the limit as $|z|\rightarrow\infty$, yields
		\begin{equation}\label{blowup2}
			\ell=\liminf_{|z| \rightarrow \infty}|z|^{n-6}|\mathcal{U}(z)|\geqslant \mu^{n-6}|\mathcal{U}(z)| \quad \mbox{for all} \quad 0<\mu<\mu^*(x).
		\end{equation}
		On the other hand, if $\mu^*(x_0)<\infty$, multiplying the identity obtained in (i) by $|z|^{n-6}$ and passing to the limit when $|z|\rightarrow\infty$, we obtain
		\begin{equation}\label{blowup3}
			\ell=\lim_{|z| \rightarrow \infty}|z|^{n-6} f(u(z))=\mu^*(x_0)^{n-6} f(x_0)<\infty.
		\end{equation}
		Finally, by \eqref{blowup2} and \eqref{blowup3}, if there exists $x_0 \in \mathbb{R}^{n}$ such that $\mu^*(x_0)<\infty$, then $\mu^*(x)<\infty$ for all $x \in \mathbb{R}^{n}$.
	\end{proof}
	
	\begin{proof}[Proof of Proposition~\ref{prop:radialsymmetry}]
		Using Lemma~\ref{lm:movingspheres}, we can apply \cite[Proposition~2.1]{MR2479025} to conclude that positive singular solutions to \eqref{ourPDE} are radially symmetric with respect to the origin. 
		The proof that they are radially monotonically decreasing will be given in the next section (See Proposition~\ref{prop:odeclassification} \ref{itm:C3}).
	\end{proof}
	
	\section{ODE analysis}\label{sec:ODEanalysis}
	This section is devoted to the classification of (positive) solutions to the problem \eqref{ourODE}. 
	The positivity will not play a role here.
	Our strategy here is inspired by the methods in \cite{MR1740359,MR3869387,MR4094467} and relies on the decomposition in Proposition~\ref{prop:decomposition}.
	It is outstanding to recover such properties for solutions to sixth order ODE with possibly chaotic dynamical behavior near equilibrium points.
	
	The next classification result will be the key part of the proof of our main theorem 
	\begin{proposition}\label{prop:odeclassification}
		Let $v \in C^{6}(\mathbb{R})$ be a solution to \eqref{ourODE}. 
		\begin{enumerate}
			\item[\namedlabel{itm:C1}{\rm (i)}] Then,
			\begin{equation}\label{estimate}
				\inf_{\mathbb{R}}|v|\leqslant a^{*}_n
			\end{equation}
			with equality holding if, and only if, $v$ is a nonzero constant function. 
			\item[\namedlabel{itm:C2}{\rm (ii)}] Conversely, if $a_0 \in(0, a^{*}_n)$, then there exists a unique $($up to translations$)$ bounded solution $v \in C^{6}(\mathbb{R})$ to \eqref{ourODE} such that $\inf_{\mathbb{R}}|v|=a_0$. This solution is periodic, has a unique local maximum and minimum per period, and is symmetric with respect to its local extrema.
			\item[\namedlabel{itm:C3}{\rm (iii)}] Moreover, if $v$ is positive, then $v^{(1)}-\gamma_{n}v<0$ in $\mathbb R$.
		\end{enumerate} 
	\end{proposition}
	
	We begin with some preliminary remarks.
	\begin{remark}
		Notice that $v\equiv0$ and $v\equiv a^{*}_n$, where $a^{*}_n>0$ is given by \eqref{constantsolutions} are exactly the only three constant solutions to \eqref{ourODE}. 
		Moreover, if $v(\cdot)$ is a solution to \eqref{ourODE},
		\begin{itemize}
			\item then $v(-\cdot)$ also is a solution to \eqref{ourODE} since it contains only even order derivatives;
			\item then $-v(\cdot)$ also is a solution to \eqref{ourODE} since $G$ is odd;
			\item then $v(\cdot+T)$ for any $T \in \mathbb{R}$ is also a solution to \eqref{ourODE} since it is autonomous.
		\end{itemize}
	\end{remark}
	
	To prove our main proposition in this section, we will need some auxiliary results quantifying the intuition that the set of bounded solutions to the higher order equation \eqref{ourODE} behaves in some respects similar to the set of solutions to a second order equation. 
	As we pointed out before, for this it is crucial that the relation \eqref{discriminant} holds. 
	
	\subsection{Boundedness}
	The following lemma states that solutions to \eqref{ourODE} are bounded, which is one of the key new results in this manuscript. 
	The proof is based solely on the conservation of energy, as so is independent of the rest of the argument.
	
	\begin{lemma}\label{lm:boundededness}
		Let $v \in C^{6}(\mathbb{R})$ be a solution to \eqref{ourODE}. Then $v$ is bounded.
	\end{lemma}
	
	\begin{proof}
		Initially,  by transforming $\tau=-t$, without loss of generality it is enough to prove that $v$ is bounded on $[0, \infty)$. 
		
		Indeed, for any solution $v \in C^{6}(\mathbb{R})$, we set
		\begin{equation}\label{criticalset}
			\mathcal Z_{+}(v)=Z_{+}^{(1)}(v)\cap Z_{+}^{(2)}(v),
		\end{equation}
		where
		\begin{equation*}
			Z_{+}^{(1)}(v)=\{t \geqslant 0: v^{(1)}(t)=0\} \quad {\rm and} \quad Z_{+}^{(2)}(v)=\{t \geqslant 0: v^{(2)}(t)=0\}.
		\end{equation*}
		
		In what follows, we have two cases to study:
		
		\noindent{\bf Case 1.} $\mathcal Z_{+}(v)$ is bounded.
		
		\noindent In this case, using Lemma~\ref{lm:asymptotics}, we get that $v$ is monotone for $t\gg1$ large enough, and thus admits a finite limit $a_{+}:=\lim_{t \rightarrow \infty}v(t)<\infty$.
		Therefore, $v$ is bounded on $[0, \infty)$.
		
		Next, we show that the other possibility cannot happen.
		
		\noindent{\bf Case 2.} $\mathcal Z_{+}(v)$ is unbounded.
		
		\noindent Now, since $F(v) \rightarrow \infty$ as $|v| \rightarrow \infty$, there exists an $R>|v(0)|$ such that $F(v)>H_{v}$ for all $|v| \geqslant R$. 
		Also, notice that $|v|<R$ in $[0, \infty)$. Indeed, by contradiction assume that $M_{R}:=\{t \geqslant 0:|v(t)| \geqslant R\}\neq\varnothing$. 
		We set $t^{*}:=\inf_{\mathbb R_+} M_{R}$. 
		
		\noindent In addition, because $|v(0)|<R$, we must have $t^{*}>0$ and $|v(t^{*})|=R$. Replacing $v(t)$ by $-v(t)$ if necessary and observing that it does not change the set $\mathcal Z_{+}(v)$, we may assume that $v(t^{*})=R$. 
		Hence, we also find $v^{(1)}(t^{*}) \geqslant 0$ and $v^{(2)}(t^{*}) \geqslant 0$. 
		Using that $\mathcal Z_{+}(v)$ is unbounded, we have that $-\infty<T:=\inf(\mathcal Z_{+}(v) \cap(t^{*}, \infty))$ is well-defined, which, by continuity of $v^{(1)}$ and $v^{(2)}$, yields $v^{(1)}(T)=v^{(2)}(T)=0$ and $\min\{v^{(1)},v^{(2)}\}\geqslant 0$ in $[t^{*}, T]$. 
		Consequently, we get $v(T) \geqslant v(t^{*})=R$, from which we deduce
		\begin{equation*}
			\mathcal{H}(v)(T)=\frac{1}{2}v^{(3)}(T)^{2}+F(v(T)) \geqslant F(v(T))>\mathcal{H}_{v},
		\end{equation*}
		which is a contradiction with the energy conservation \eqref{conservationofenergy} in Proposition~\ref{prop:conservation}.
		
		The proof is finished.
	\end{proof}
	
	\subsection{Comparison}
	The first result states that any bounded entire solution to \eqref{ourODE} is uniquely determined by only two (instead of $6$) initial values.
	
	\begin{lemma}\label{lm:comparisonprinciple}
		Let $v_1,v_2\in C^{6}(\mathbb{R})$ be bounded solutions to \eqref{ourODE}.
		If $v_1(0)=v_2(0)$ and $v_1^{(1)}(0)=v_2^{(1)}(0)$. 
		Then $v_1 \equiv v_2$.
	\end{lemma}
	
	\begin{proof}
		Let $v_1,v_2\in C^{6}(\mathbb{R})$ be bounded solutions to \eqref{ourODE} which satisfy the initial conditions $v_1(0)=v_2(0)$ and $v_1^{(1)}(0)=v_2^{(1)}(0)$. 
		By interchanging $v_1$ and $v_2$ or replacing $v_1(t)$ and $v_2(t)$ by $v_1(-t)$ and $v_2(-t))$, we may assume without loss of generality that $v_1^{(j)}(0) \geqslant v_2^{(j)}(0)$ for $j=2,3,4,5$.
		
		Suppose, by contradiction, that $v_1 \not \equiv v_2$. 
		Then, by standard uniqueness result for ODEs, we have that $v_1^{(j)}(0) \neq v_2^{(j)}(0)$ for $j=2,3,4,5$. 
		In both cases, we deduce from our hypotheses on the initial conditions $v(t)>w(t)$ on $(0, \tau)$ for some sufficiently small $0<\tau\ll1$.
		
		Let $\{\lambda_1,\lambda_2,\lambda_3\}\subset\mathbb R$ be given by \eqref{indicialroots}. Using the decomposition in Proposition~\ref{prop:decomposition}, let us set the auxiliary functions $\{\Phi_1,\Phi_2,\Phi_3\}\subset C^6(\mathbb R)$ defined as
		\begin{equation}\label{auxiliaryfunctions}
			\Phi_1:=L_{\lambda_1}, \quad \Phi_2:=L_{\lambda_1}\circ L_{\lambda_2}, \quad \mbox{and} \quad \Phi_3:=L_{\lambda_1}\circ L_{\lambda_2}\circ L_{\lambda_3}.
		\end{equation}
		By the contradiction assumption, we have
		\begin{equation}\label{comparison1}
			(\Phi_2(v_1)-\Phi_2(v_2))(0) \geqslant 0 \quad \text { and } \quad(\Phi_2(v_1)-\Phi_2(v_2))^{(1)}(0) \geqslant 0
		\end{equation}
		and 
		\begin{equation}\label{comparison3}
			(\Phi_1(v_1)-\Phi_1(v_2))(0) \geqslant 0 \quad \text { and } \quad(\Phi_1(v_1)-\Phi_1(v_2))^{(1)}(0) \geqslant 0.
		\end{equation}
		Also, using that $v_1$ and $v_2$ satisfy \eqref{ourODE}, we find
		\begin{equation*}
			-L_{\lambda_3}(\Phi_2(v_1)-\Phi_2(v_2))=f(v_2(t))-f(v_1(t)) \quad \text {for all} \quad t \in \mathbb{R}.
		\end{equation*}
		
		Next, using that $v_1(t)>v_2(t)$ on $(0, \tau)$ and the strictly monotonicity of since the function $u \mapsto f(u)$, we get
		\begin{equation}\label{comparison2}
			-L_{\lambda_3}(\Phi_2(v_1)-\Phi_2(v_2))(t)>0 \quad \text { for all } \quad t \in(0, \tau).
		\end{equation}
		Hence, from \eqref{comparison1} and \eqref{comparison2}, we easily find that $(\Phi_2(v_1)-\Phi_2(v_2))(t) \geqslant 0$ for $t\in(0, \tau)$, or equivalently,
		\begin{equation}\label{comparison4}
			-L_{\lambda_2}(\Phi_1(v_1)-\Phi_1(v_2))(t)=\Phi_2(v_1)-\Phi_2(v_2)(t) \geqslant0 \quad \text { for all } \quad t \in(0, \tau).
		\end{equation}
		Whence, from \eqref{comparison3} and \eqref{comparison4}, we easily find that $(\Phi_
		1(v_1)-\Phi_1(v_2))(t) \geqslant 0$ for $t\in(0, \tau)$, or equivalently,
		\begin{equation}\label{comparison5}
			(v_1-v_2)^{(2)}(t) \geqslant \lambda_1(v_1-v_2)(t)>0 \quad \text { for all } \quad t \in(0, \tau).
		\end{equation}
		By the hypotheses, we have $(v_1-v_2)^{(1)}(0) \geqslant 0$, which combined with $\lambda_3>0$ and \eqref{comparison5} implies that $(v_1-v_2)^{(1)}(t)>0$ for all $t \in(0, \tau)$. 
		Thus, $v_1-v_2$ is strictly increasing on $(0, \tau)$, and since $0<\tau\ll1$ was arbitrarily small such that $v_1-v_2>0$ on $(0, \tau)$, we conclude that $v_1-v_2>0$ for all $t\in\mathbb R$.
		
		Repeating the above arguments for the interval $(0, \infty)$ instead of $(0, \tau)$, we get from \eqref{comparison5} that $(v_1-v_2)^{(1)}$ is positive and strictly increasing on $(0, \infty)$, which clearly contradicts the boundedness of $v_1-v_2$. 
		This shows that $v_1 \equiv v_2$, and concludes the proof of this lemma.
	\end{proof}
	
	As a consequence, we have the following symmetry result
	\begin{corollary}\label{cor:symmetry}
		Let $v \in C^{6}(\mathbb{R})$ be a bounded solution to \eqref{ourODE}.
		\begin{itemize}
			\item[{\rm (i)}] If $v^{(1)}(t_{0})=0$ for some $t_{0} \in \mathbb{R}$, then $v$ is symmetric with respect to $t_{0}$, that is $v(t_{0}+t)=v(t_{0}-t)$ for all $t \in \mathbb{R}$;
			\item[{\rm (ii)}] If $v(t_{0})=0$ for some $t_{0} \in \mathbb{R}$, then $v$ is anti-symmetric with respect to $t_{0}$, that is, $v(t_{0}-t)=-v(t_{0}+t)$ for all $t \in \mathbb{R}$. 
		\end{itemize}
	\end{corollary}
	
	\begin{proof}
		To prove (i), notice that since \eqref{ourODE} is autonomous, we may assume $t_{0}=0$. 
		Moreover, if $v$ is a solution to \eqref{ourODE}, then $\hat{v}(t):=v(-t)$ also is. 
		Hence $v(0)=\hat{v}(0)$ and, by assumption, $v^{(1)}(0)=\hat{v}^{(1)}(0)=0$, which by Lemma~\ref{lm:comparisonprinciple} gives $v \equiv \hat{v}$.
		
		To prove (ii), we use the same strategy.
	\end{proof}
	
	With a similar proof, we can also prove the following comparison
	\begin{lemma}\label{lm:weakcomparisonprinciple}
		Let $v_1,v_2\in C^{6}(\mathbb{R})$ be bounded solutions to \eqref{ourODE}.
		Suppose that 
		\begin{align*}
			v_1(0)&\geqslant v_2(0)\\
			v_1^{(1)}(0)&\geqslant v_2^{(1)}(0)\\
			L_{\lambda_1}(v_1(0))&\geqslant L_{\lambda_1}(v_2(0))\\ L_{\lambda_1}^{(1)}(v_1(0))&\geqslant L_{\lambda_1}^{(1)}(v_2(0))\\
			(L_{\lambda_1}\circ L_{\lambda_2})(v_1(0))&\geqslant (L_{\lambda_1}\circ L_{\lambda_2})(v_2(0))\\ 
			(L_{\lambda_1}\circ L_{\lambda_2})^{(1)}(v_1(0))&\geqslant (L_{\lambda_1}\circ L_{\lambda_2})^{(1)}(v_2(0)).
		\end{align*}
		Then $v_1 \equiv v_2$.
	\end{lemma}
	
	\begin{proof}
		Follows the same strategy as in Lemma~\ref{lm:comparisonprinciple} 
	\end{proof}
	
	\subsection{Asymptotics}
	This subsection is devoted to the study of some asymptotic properties for solutions to \eqref{ourODE}.
	
	First, it is convenient to introduce the following definition
	\begin{definition}
		For any $v\in C^{6}(\mathbb R)$ solution to \eqref{ourODE}, let us define its {asymptotic set}  by
		\begin{equation*}
			\mathcal{A}(v):=\left\{\ell\in \mathbb R\cup\{\pm\infty\} : \lim_{t\rightarrow\pm\infty}v(t)=\ell\right\}.
		\end{equation*}
		In other words, $\mathcal{A}(v)$ is the set of all possible limits at infinity of  $v$.
	\end{definition}
	
	The content of the following lemma shows that \eqref{ourODE} does not possess a solution that tends to either plus or minus infinity at infinity, that is, a solution that blows up does it in finite time.
	\begin{lemma}\label{lm:asymptotics}
		Let $v \in C^{6}(\mathbb{R})$ be a solution to \eqref{ourODE}. If $\ell_{\pm}:=\lim_{t \rightarrow \pm\infty} v(t) \in \mathbb{R} \cup\{\pm \infty\}$ exists, then $\ell_{\pm} \in \mathbb{R}$. 
	\end{lemma}
	
	\begin{proof}
		Initially, notice that by interchanging $v(t)$ by $v(-t)$, it is only necessary to consider the case $t\rightarrow+\infty$.
		
		Suppose by contradiction that the lemma does not hold. 
		Thus, one of the following two possibilities shall happen: either the asymptotic limit of $v$ is a finite constant $\ell_+>0$, which does not belong to the asymptotic set $\mathcal{A}$, or the limit blows up, that is, $\ell_+=+\infty$. 
		
		Let us consider these two cases separately:
		
		\noindent{\bf Case 1:} $\ell_+\in [0,\infty)\setminus\{0,a^{*}_n\}$. 
		
		\noindent By assumption, we have 
		\begin{equation}\label{bound1}
			\displaystyle\lim_{t\rightarrow\infty}\left(K_0v(t)-f(v(t))\right)=\kappa, \quad \hbox{where} \quad \kappa:=K_0\ell_+-f(\ell_+)\neq0,
		\end{equation}
		which implies
		\begin{equation}\label{bound2}
			K_0v(t)-f(v(t))=v^{(6)}(t)-K_4v^{(4)}(t)+K_2v^{(2)}(t).
		\end{equation}
		Hence, combining \eqref{bound1} with \eqref{bound2} implies that for any $\varepsilon>0$ there exists $T\gg1$ sufficiently large satisfying
		\begin{equation}\label{bound3}
			\kappa-\varepsilon<v^{(6)}(t)-K_4v^{(4)}(t)+K_2v^{(2)}(t)<\kappa+\varepsilon.
		\end{equation}
		
		Now, integrating \eqref{bound3}, we obtain
		\begin{equation*}
			\int_{T}^{t}(\kappa-\varepsilon)\ud\tau<\int_{T}^{t}\left[v^{(6)}(\tau)-K_4v^{(4)}(\tau)+K_2v^{(2)}(\tau)\right]\ud\tau<\int_{T}^{t}(\kappa+\varepsilon)\ud\tau,
		\end{equation*}
		which yields
		\begin{equation}\label{bound5}
			(\kappa-\varepsilon)(t-T)+C_1(T)<v^{(5)}(t)-K_4v^{(3)}(t)+K_2v^{(1)}(t)<(\kappa+\varepsilon)(t-T)+C_1(T),
		\end{equation}
		for some $C_1(T)>0$. 
		Defining $\delta:=\sup_{t\geqslant T}|v(t)-v(T)|<\infty$, we obtain
		\begin{equation*}
			\left|\int_{T}^{t}K_4v^{(3)}(\tau)\ud\tau\right|\leqslant|K_4|\delta \quad {\rm and} \quad \left|\int_{T}^{t}K_2v^{(1)}(\tau)\ud\tau\right|\leqslant|K_2|\delta,
		\end{equation*}
		which, by integrating $\eqref{bound5}$, gives us
		\begin{equation}\label{bound10}
			\frac{(\kappa-\varepsilon)}{2}(t-T)^2+L(t)<v^{(4)}(t)<\frac{(\kappa+\varepsilon)}{2}(t-T)^2+R(t),
		\end{equation}
		where $L(t),R(t)\in \mathcal{O}(t^2)$, namely
		\begin{equation*}
			L(t)=C_1(T)(T-t)-|K_4|\delta-|K_2|\delta+C_2(T)
		\end{equation*}
		and
		\begin{equation*}
			R(t)=C_1(T)(T-t)+|K_4|\delta+|K_2|\delta+C_2(T),
		\end{equation*}
		for some $C_2(T)>0$. 
		
		Then, repeating the same integration procedure in \eqref{bound10}, we find
		\begin{equation}\label{bound11}
			\frac{(\kappa-\varepsilon)}{6!}(t-T)^6+\mathcal{O}(t^6)<v(t)<\frac{(\kappa+\varepsilon)}{6!}(t-T)^6+\mathcal{O}(t^6) \quad \mbox{as} \quad t\rightarrow\infty.
		\end{equation}
		Therefore, since $\kappa\neq0$ we can choose $0<\varepsilon\ll1$ sufficiently small such that $\kappa-\varepsilon$ and $\kappa+\varepsilon$ have the same sign. Finally, by passing to the limit as $t\rightarrow\infty$ on inequality \eqref{bound11}, we obtain that $v$ blows-up and $\ell_+=\infty$, which is contradiction.
		This concludes the proof of the first case.
		
		The second case requires a suitable choice of test functions, which is inspired in \cite{MR1879326}. 
		
		\noindent{\bf Case 2:} $\ell_+=+\infty$. 
		
		\noindent Let $\phi_0\in C^{\infty}({[0,\infty]})$ be a nonnegative function satisfying $\phi_0>0$ in $[0,2)$, 
		\begin{equation*}
			\phi_0(z)=
			\begin{cases}
				1,&\ {\rm for} \ 0\leqslant z\leqslant1,\\
				0,&\ {\rm for} \ z\geqslant2,
			\end{cases}
		\end{equation*}
		and, let us fix the positive constants
		\begin{equation}\label{miti-pokh}
			M_j:=\int_{0}^{2}\frac{|\phi_0^{(j)}(z)|}{|\phi_0(z)|}\, \ud z \quad {\rm for} \quad j=1,2,3,4,5,6.
		\end{equation}
		Using the contradiction assumption, we may assume that there exists $T>0$ such that for $t>T$, it follows
		\begin{equation}\label{blow1}
			v^{(6)}(t)-K_4v^{(4)}(t)+K_2v^{(2)}(t)=K_0v(t)-f(v(t))\geqslant \frac{1}{2}f(v(t))
		\end{equation}
		and
		\begin{equation}\label{blow2}
			v^{(5)}(t)-K_4v^{(3)}(t)+K_2v^{(1)}(t)=\frac{1}{2}\int_{T}^{t}f(v(\tau))\ud\tau+C_1(T).
		\end{equation}
		Besides, as a consequence of \eqref{blow2}, we can find $T^{*}>T$ satisfying
		\begin{equation*}
			v^{(5)}(T^{*})-K_4v^{(3)}(T^{*})+K_2v^{(1)}(T^{*}):=\upsilon>0.
		\end{equation*}
		Furthermore, since \eqref{ourODE} is autonomous, we may suppose without loss of generality that $T^{*}=0$.
		Then, multiplying inequality \eqref{blow1} by  $\phi(t)=\phi_0(\tau/t)$, and by integrating, we find
		\begin{equation*}
			\int_{0}^{T'}\left[v^{(6)}(\tau)\phi(\tau)-K_4v^{(4)}(\tau)\phi(\tau)+K_2v^{(2)}(\tau)\phi(\tau)\right]\ud\tau\geqslant\frac{1}{2}\int_{0}^{T'}v(\tau)\ud\tau,
		\end{equation*}
		where $T'=2T$. Moreover, integration by parts combined with $\phi^{(j)}(T')=0$ for $j=0,1,2,3,4,5$ implies 
		\begin{equation}\label{blow4}
			\int_{0}^{T'}\left[v(\tau)\phi^{(6)}(\tau)-K_4v(\tau)\phi^{(4)}(\tau)+K_2v(\tau)\phi^{(2)}(\tau)\right]\ud\tau\geqslant\frac{1}{2}\int_{0}^{T'}f(v(\tau))\ud\tau+\upsilon.
		\end{equation}
		On the other hand, applying the Young inequality on the right-hand side of \eqref{blow4}, it follows
		\begin{equation}\label{blow5}
			v(\tau)|\phi^{(j)}(\tau)|=\varepsilon v^{\frac{n+6}{n-6}}(\tau)\phi(\tau)+C_{\varepsilon}\frac{|\phi^{(j)}(\tau)|^{\frac{n+6}{12}}}{\phi(\tau)^{\frac{n-6}{12}}} \quad {\rm for} \quad j=0,1,2,3,4,5.
		\end{equation}
		Hence, combining \eqref{blow5} and \eqref{blow4}, we have that for $0<\varepsilon\ll1$ sufficiently small, it follows that there exists $\widetilde{C}_1>0$ satisfying
		\begin{equation*}
			\widetilde{C}_1\int_{0}^{T'}\left[\frac{|\phi^{(6)}(\tau)|^{\frac{n+6}{12}}}{\phi(\tau)^{\frac{n-6}{12}}}+\frac{|\phi^{(4)}(\tau)|^{\frac{n+6}{12}}}{\phi(\tau)^{\frac{n-6}{12}}}+\frac{|\phi^{(2)}(\tau)|^{\frac{n+6}{12}}}{\phi(\tau)^{\frac{n-6}{12}}}\right]\ud\tau\geqslant\frac{1}{6}\int_{0}^{T'}f(v(\tau))\ud\tau+\upsilon.
		\end{equation*}
		Now, by \eqref{miti-pokh}, one can find $\widetilde{C}_2>0$ such that
		\begin{equation}\label{blow7}	                               \widetilde{C}_2\left(M_6T^{-\frac{n+4}{2}}-M_4T^{-\frac{n+3}{3}}+M_2T^{-\frac{n}{6}}\right)\geqslant\frac{1}{6}\int_{0}^{T}f(v(\tau))\ud\tau.
		\end{equation}
		Therefore, passing to the limit in \eqref{blow7} the left-hand side converges, whereas the right-hand side blows-up; this is a contradiction. 
		
		For proving the second part, let us notice that 
		\begin{equation*}
			\lim_{t\rightarrow\infty}\mathcal{H}(t,v)=\left(F(\ell_+)-\frac{K_0}{2}\ell_+^{2}\right)\geqslant0,
		\end{equation*}
		which implies $\ell^{*}=0$ and $\mathcal{H}(v)=0$.
	\end{proof}
	
	\begin{lemma}\label{lm:maximumandminimumpoints}
		Let $v\in C^{6}(\mathbb R)$ be a positive solution to \eqref{ourODE}.
		Assume that $v$ is a non-constant function, and denote by $\{s_{k}\}_{k\in\mathbb{N}}$ and $\{S_{k}\}_{k\in\mathbb{N}}$ its sets of local minimum and maximum points of $v$, respectively. 
		Then, it holds that  $v(S_k)>a^{*}_n$ and $v(s_{k})<a^{*}_n$ for any $k\in\mathbb N$.
	\end{lemma}
	
	\begin{proof}
		Let us assume by contradiction that there exists a maximum point $t=t_{0}$ of $v$ such that $v(t_{0}) \leqslant a^{*}_n$ and $v^{(1)}(t_{0})=0$, $v^{(2)}(t_{0}) \leqslant 0$. 
		Notice that from \eqref{ourODE}, we have two possibilities which we describe as follows
		
		\noindent{\bf Case 1.} $v(t_{0})=a^{*}_n$, $v^{(2)}(t_{0})=0$, $v^{(4)}(t_{0})=0$, and $v^{(6)}(t_{0})=0$.
		
		\noindent Initially, we observe that $v^{(3)}(t_{0})=0$ and $v^{(5)}(t_{0})=0$. 
		Otherwise, without loss of generality, we may suppose that $v^{(3)}(t_{0})>0$ or $v^{(5)}(t_{0})>0$.
		If, $v^{(3)}(t_{0})>0$, by an elementary analysis, we deduce that there exists $\delta>0$ satisfying $v^{(2)}(t)>0$ and $v^{(1)}(t)>0$ in $(t_{0}, t_{0}+\delta)$.
		This contradicts the fact the point $t=t_{0}$ is a local maximum of $v(t)$.
		If, $v^{(5)}(t_{0})>0$, we can apply the same strategy.
		Then, by the uniqueness, we get that $v(t) \equiv a^{*}_n$, which is impossible.
		
		\noindent{\bf Case 2.} $v^{(6)}(t_{0})=K_{4}v^{(4)}(t_{0})-K_{2}v^{(2)}(t_{0})-g(v(t_{0}))<0$.
		
		\noindent Now, notice that for some small $\varepsilon>0$ it follows $v^{(6)}(t)<0$ for $t \in(t_{0}, t_{0}+\varepsilon)$ . 
		In this case, we assume $v^{(3)}(t_{0}) \leqslant 0$ and $v^{(5)}(t_{0}) \leqslant 0$, which gives us $v^{(1)}(t)<0$, $v^{(3)}(t)<0$ and $v^{(5)}(t)<0$ for $t \in(t_{0}, t_{0}+\varepsilon)$. 
		Now, setting
		\begin{equation*}
			t_{1}=\sup\left\{\tilde{t}>t_{0} : v^{(1)}(t)<0, \ v^{(3)}(t)<0, \ {\rm and} \ t\in(t_{0}, \tilde{t})\right\}.
		\end{equation*}
		we have that since $v$ oscillates, it follows that $t_{1}<\infty$ and either $v^{(1)}(t_{1})=0$, $v^{(3)}(t_{1})=0$, or $v^{(5)}(t_{1})=0$ because of continuity. 
		This implies that $\mathcal{E}_1(v(t_{1}))v^{(1)}(t_{1})=0$. 
		Moreover, since $v^{(2)}$ is decreasing in $(t_{0}, t_{1})$, combined with $v^{(2)}(t_{0}) \leqslant 0$, we obtain that $H_2(v)v^{(2)}$ is increasing in $(t_{0}, t_{1})$.
		Putting these together and remembering that
		$v^{(1)}(t_{0})=0$, it follows from Proposition~\ref{prop:conservation} that the equality holds
		\begin{equation*}
			\frac{1}{2}{v^{(3)}}(t_{1})^2-\frac{K_{2}}{2}v^{(1)}(t_{1})^2+G(v(t_{1}))=\frac{1}{2}{v^{(3)}}(t_{0})^2+G(v(t_{0})),
		\end{equation*}
		which yields
		\begin{equation*}
			G(v(t_{0}))-G(v(t_{1}))=\frac{1}{2}({v^{(3)}}(t_{1})-{v^{(3)}}(t_{0}))^2-\frac{K_{2}}{2}v^{(1)}(t_{1})^2>0.
		\end{equation*}
		On the other hand, since $v(t_{1})<v(t_{0}) \leqslant a^{*}_n$, it is straightforward to see that $G(v(t_{1}))>G(v(t_{0}))$; this contradicts the last identity.
		The proof of the first case is finished.
		
		Similarly, we can prove the second part for local minimum points of $v$, and this concludes the proof of the lemma.
	\end{proof}
	
	The following lemma shows that when a solution tends to a limit monotonically, then it has to be an equilibrium point. 
	\begin{lemma}\label{lm:energylimits}
		Let $v \in C^{6}(\mathbb{R})$ be a bounded solution to \eqref{ourPDE}. If $v$ is eventually monotone at infinity, that is, $\lim_{t \rightarrow \pm\infty}{\rm sign}(v^{(1)}(t))\neq0$, then
		\begin{equation*}
			\lim_{t \rightarrow \pm\infty} v(t) \in\{0, \pm a^{*}_n\} \quad {\rm and } \quad \lim_{t \rightarrow \pm\infty} v^{(j)}(t)=0 \quad {\rm for} \quad j=1,\dots,5.
		\end{equation*}
		In other terms,
		\begin{equation*}
			\mathcal{A}(v)=\{0, \pm a^{*}_n\} \quad {\rm and} \quad \mathcal{A}(v^{(j)})=\{0\} \quad {\rm for} \quad j=1,\dots,5.
		\end{equation*}
	\end{lemma}
	
	\begin{proof}
		By interchanging $v(t)$ by $v(-t)$, we may assume that $\lim_{t \rightarrow +\infty}{\rm sign}(v^{(1)}(t))\neq0$. 
		Then, there exists the limit
		\begin{equation*}
			\lim_{t\rightarrow +\infty} v(t):=\ell_{0}
		\end{equation*}
		and $v$ increases towards $\ell_{0}$ as $t \rightarrow +\infty$. Since $v$ is bounded, $\ell_{0}<+\infty$.

		\noindent{\bf Claim 1:} $\lim_{t\rightarrow}v^{(4)}(t)=0$ and $\lim_{t\rightarrow}v^{(2)}(t)=0$ as $t\rightarrow+\infty$.
		
		\noindent Let us define the function $\Upsilon_1(v)=v^{(5)}-K_4v^{(3)}+K_2v^{(1)}$, which satisfies $\Upsilon_1^{(2)}(v)=v^{(1)}g(v)$.
		Now we have to consider two cases:
		
		\noindent{\bf Case 1:} $g^{(1)}(\ell_{0}) \neq 0$
		
		\noindent  
		In this case, $g^{(1)}(v)$ has a sign for $T\gg1$ large enough, that is, either $g^{(1)}(v) \geqslant 0$ or $g^{(1)}(v) \leqslant 0$ for large $T\gg1$.
		In addition, since $v^(1) \geqslant 0$, it follows $\lim_{t\rightarrow\infty} {\rm sign}(\Upsilon_1^{(2)}(v(t)))\neq0$, which implies that $\Upsilon_1^{(2)}(v(t))$ has a sign for $T\gg1$ large enough; thus by a comparison principle $\Upsilon_1(v(t))$ also does. 
		Furthermore, because $\Upsilon_1(v(t))=v^{5}(t)-K_4v^{(3)}(t)+K_2v^{(1)}(t)$ has a sign for $T\gg1$ large enough, we find the the following limit exists:
		\begin{equation*}
			\lim_{t \rightarrow +\infty}\left(v^{(4)}(t)-K_4v^{(2)}(t)+K_2v(t)\right):=\ell_{1}.
		\end{equation*}
		Consequently, we get
		\begin{equation}\label{limits1}
			\lim_{t \rightarrow +\infty}\left(v^{(4)}(t)-K_4v^{(2)}(t)\right):=\ell_{1}-K_2\ell_0.
		\end{equation}
		Defining $\Upsilon_{2}(v):=v^{(2)}-K_4v$, we get $\lim_{t \rightarrow +\infty}\Upsilon_{2}^{(2)}(v)\neq0$, that is, $\Upsilon_{2}(v(t))$ has a sign for $T\gg1$ large enough.
		Hence, by a comparison principle $\Upsilon_2(v(t))$ also does. 
		Moreover, since $\Upsilon_2(v(t))=v^{(2)}(t)-K_4v(t)$ has a sign for $T\gg1$ large enough, we get that the the following limit exists
		\begin{equation*}
			\lim_{t \rightarrow +\infty}\left(v^{(2)}(t)-K_4v(t)\right):=\ell_{2}.
		\end{equation*}
		Therefore, we conclude $v^{(2)}(t)\rightarrow \ell_{2}+K_4\ell_{0}$ as $t \rightarrow \infty$, which by the boundedness of $v$, proves that $v^{(2)}(t)\rightarrow0$ as $t\rightarrow+\infty$.
		Finally, going back to \eqref{limits1}, we find $v^{(4)}(t)\rightarrow0$ as $t\rightarrow+\infty$.
		
		\noindent{\bf Case 2:} $g^{(1)}(\ell_{0})=0$
		
		\noindent In this case, we can define $\widetilde{\Upsilon}_1(v)=v^{(5)}-K_4v^{(3)}+\frac{K_2}{2}v^{(1)}$ and $\widetilde{\Upsilon}_2(v)=v^{(4)}-\frac{K_4}{2}v^{(2)}$ and proceed as before to conclude the proof.
		
		We finished the proof of the first claim
		
		At last, using that $v^{(2)}(t) \rightarrow 0$ and $v^{(4)}(t) \rightarrow 0$ as $t\rightarrow+\infty$, one can prove that $v^{(1)}(t) \rightarrow 0$ and $v^{(3)}(t) \rightarrow 0$ as $t\rightarrow+\infty$.
		Since, we can write \eqref{ourODE} as
		\begin{equation*}
			v^{(6)}=K_{4}v^{(4)}-K_{2}v^{(2)}-g(v) \quad {\rm in} \quad \mathbb R,
		\end{equation*}
		it holds
		\begin{equation*}
			\lim_{x \rightarrow +\infty} v^{(6)}(t):=\ell_3=g(\ell_{0}).
		\end{equation*}
		Therefore, $v$ is bounded, we find $\ell_3=0$, and thus $\ell_{0} \in \mathcal{A}(v)$. 
		By this discussion, it is easy now to see that $\mathcal{A}(v^{(j)})=0$ for $j=1,2,3,4,5$, which concludes the proof of the lemma.
	\end{proof}
	
	\subsection{Ordering}
	Now we prove that, as in the lower order cases, the Hamiltonian energy is a parameter that orders bounded solutions in the $(v, v^{(1)})$-phase plane.
	
	Our first technical lemma is a sign property for some terms of the Hamiltonian energy associated to \eqref{ourODE}, which will be important in our upcoming arguments.
	Although, our results ares similar in spirit to the one in \cite[Lemma 5]{MR1740359}, we provide a shorter proof, which also works in our sixth order situation.
	\begin{lemma}\label{lm:signproperty}
		Let $v\in C^{6}(\mathbb R)$ be a bounded solution to \eqref{ourODE}. The following sign identities hold
		\begin{itemize}
			\item[{\rm (i)}] ${\rm sign}(v^{(1)})={\rm sign}(\mathcal{E}_1(v))$;
			\item[{\rm (ii)}] ${\rm sign}(v^{(2)})={\rm sign}(\mathcal{E}_2(v))$.
		\end{itemize}
		where $\mathcal{E}_1(v),\mathcal{E}_2(v)$ the auxiliary functions given by \eqref{auxiliarysignfunctions}.
	\end{lemma}
	
	\begin{proof}
		The proof will be divided into three steps.
		
		\noindent {\bf Step 1.} There exist $\alpha_1,\alpha_2>0$ (depending only on $K_2,K_4$) such that
		\begin{equation*}
			0<\Psi_{1,\alpha_1}(v)v^{(1)}<\mathcal{E}_1(v)v^{(1)} \quad {\rm and} \quad 0<\Psi_{2,\alpha_2}(v)v^{(2)}<\mathcal{E}_2(v)v^{(2)},
		\end{equation*}
		where
		\begin{equation*}
			\Psi_{1,\alpha_1}(v)=v^{(5)}-{K_4}v^{(3)}+\alpha_1 v^{(1)} \quad {\rm and} \quad \Psi_{2,\alpha_2}(v)=-v^{(4)}+\alpha_2 v^{(2)}.
		\end{equation*}
		
		Notice that since
		\begin{equation*}
			\mathcal{E}_1(v)v^{(1)}= \Psi_{1,\alpha_1}(v)v^{(1)}-\sqrt{\left(\frac{K_2}{2}-\alpha_1\right)}{v^{(1)}}^2,
		\end{equation*}
		and
		\begin{equation*}
			\mathcal{E}_2(v)v^{(2)}= \Psi_{2,\alpha_2}(v)v^{(2)}-\sqrt{\left(\frac{K_4}{2}-\alpha_2\right)}{v^{(2)}}^2,
		\end{equation*}
		there exists  $C_1,C_2\geqslant0$ (depending only on $K_0,K_2$) such that 
		\begin{equation}\label{ordering8}
			\mathcal{E}_1(v)v^{(1)}= \Psi_{1,\alpha_1}(v)v^{(1)}+\sqrt{C_1}{v^{(1)}}^2> \Psi_{1,\alpha_1}(v)v^{(1)},
		\end{equation}
		and
		\begin{equation}\label{ordering4}
			\mathcal{E}_2(v)v^{(2)}= \Psi_{2,\alpha_2}(v)v^{(2)}+\sqrt{C_2}{v^{(2)}}^2> \Psi_{2,\alpha_2}(v)v^{(2)},
		\end{equation}
		which proves the first claim.
		
		Now, we are left to study the sign of these auxiliary terms. 
		
		Let us start with the term related to the first derivative. 
		For this, we observe that $\Psi_{1,\alpha_1}(v)$ can be written as a positive fourth order operator on $v^{(1)}$.
		
		\noindent {\bf Step 2.} ${\rm sign}(v^{(1)})={\rm sign}(\Psi_{1,\alpha_1}(v))$.
		
		We divide the remaining part of the proof into will be divided into two cases.
		
		\noindent{\bf Case 1.} If $v^{(1)}\geqslant0$, then $\Psi_{1,\alpha_1}(v)\geqslant0$.
		
		\noindent Let $t_{0} \in \mathbb{R}$ be arbitrary. 
		We may assume that $v^{(1)}(t_{0}) \geqslant 0$. 
		We see from Corollary~\ref{cor:symmetry} that Claim 2 holds
		if $v^{(1)}(t_{0})=0$. 
		We thus assume that $v^{(1)}(t_{0})>0$. 
		Since $v$ is bounded there exist $-\infty \leqslant t_{1}<t_{0}<t_{2} \leqslant \infty$ such that $v^{(1)}(t_{1})=v^{(1)}(t_{2})=0$ and $v^{(1)}<0$ in $(t_1,t_2 )$. By Corollary~\ref{cor:symmetry} and Lemma~\ref{lm:energylimits}, we get that $v^{(3)}(t_{1})=v^{(3)}(t_{2})=0$, which implies  $\Psi_{1,\alpha_1}(v(t_{1}))=\Psi_{1,\alpha_1}(v(t_{2}))=0$.
		Thus, we arrive at the following system
		\begin{equation}\label{auxiliarsystem1}
			\begin{cases}
				(L_{\alpha_1^+}\circ L_{\alpha_1^-})(v^{(1)})=\Psi_{1,\alpha_1}(v) \quad {\rm in} \quad (t_1,t_2).\\
				L_{\alpha^-_1}(v^{(1)}(t_{1}))=L_{\alpha^-_1}(v^{(1)}(t_{2}))=v^{(1)}(t_1)=v^{(1)}(t_2)=0,
			\end{cases} 
		\end{equation}
		where $\alpha_1^++\alpha_1^-=K_4$ and $\alpha_1^+\alpha_1^-=\alpha_1$.
		
		Next, suppose by contradiction that $\Psi_{1,\alpha_1}(v(t_0))\leqslant0$.
		Hence, using the maximum principle in \eqref{auxiliarsystem1}, it follows that $v^{(1)}(t_0)<0$, which is a contradiction.
		The first case is proved.
		
		For the second case, the idea is similar.
		
		\noindent{\bf Case 2.} If $v^{(1)}<0$, then $\Psi_{1,\alpha_1}(v)<0$.
		
		\noindent 
		When $v^{(1)}(t_0)<0$, as before one can find $-\infty \leqslant t_{1}<t_{0}<t_{2} \leqslant \infty$ such that $v^{(1)}<0$ in $(t_1,t_2 )$ and \eqref{auxiliarsystem1} holds.
		Using the same contradiction argument, suppose that there exists $t_0\in(t_1,t_2)$ such that $\Psi_{1,\alpha_1}(v(t_0))\geqslant0$.
		By the minimum principle applied to \eqref{auxiliarsystem1}, we find $v^{(1)}(t_0)\geqslant 0$, which is also a contradiction.
		This finishes the proof of the second claim.
		
		The study of the remaining term is less complicated and direct application of a comparison principle.
		Indeed, notice that $\Psi_{2,\alpha_2}(v)$ can be written as a positive second order operator on $v^{(2)}$.
		
		\noindent{\bf Step 3.} ${\rm sign}(v^{(2)})={\rm sign}(\Psi_{2,{\alpha_2}}(v))$.
		
		\noindent As before, we must study two cases as follows.
		
		\noindent{\bf Case 1.} If $v^{(2)}\geqslant0$, then  $\Psi_{2,{\alpha_2}}(v)\geqslant0$.
		
		\noindent In fact, we may assume that $v^{(2)}(t_{0}) \geqslant 0$.
		We know from Corollary~\ref{cor:symmetry} that Claim 2 holds
		if $v^{(2)}(t_{0})=0$. 
		We thus assume that $v^{(2)}(t_{0})>0$. 
		Since $v$ is bounded there exist $-\infty \leqslant t_{1}<t_{0}<t_{2} \leqslant \infty$ such that $v^{(2)}(t_{1})=v^{(2)}(t_{2})=0$ and $v^{(2)}>0$ in $(t_1,t_2 )$. 
		By Corollary~\ref{cor:symmetry} and Lemma~\ref{lm:energylimits}, we get that $v^{(4)}(t_{1})=v^{(4)}(t_{2})=0$, which implies $L_{\alpha_2}(v^{(2)}(t_1))= L_{\alpha_2}(v^{(2)}(t_2))=0$.
		
		Now, it is convenient to write
		\begin{equation}\label{auxiliarsystem2}
			\begin{cases}
				-L_{\alpha_2}(v^{(2)})=\Psi_{2,{\alpha_2}}(v) \quad {\rm in} \quad (t_1,t_2)\\
				L_{\alpha_2}(v^{(2)}(t_1))= L_{\alpha_2}(v^{(2)}(t_2))=0.
			\end{cases}
		\end{equation}
		It is straightforward to check that $\Psi_{2,{\alpha_2}}(v(t_0)>0$.
		Otherwise, we would have $\Psi_{2,{\alpha_2}}(v(t_0)<0$.
		Thus, by the maximum principle, $v^{(2)}(t_{0})<0$, which contradicts the assumption.
		
		The second case has a similar proof.
		
		\noindent{\bf Case 2.} If $v^{(2)}<0$, then $\Psi_{2,{\alpha_2}}(v)<0$.
		
		\noindent When $v^{(2)}(t_0)<0$, as before one can find $-\infty \leqslant t_{1}<t_{0}<t_{2} \leqslant \infty$ such that $v^{(2)}<0$ in $(t_1,t_2 )$ and \eqref{auxiliarsystem2} holds.
		Again, the same argument by contradiction based on the minimum principle yields $\Psi_{2,\alpha_2}(v(t_0))<0$.
		
		The conclusion is a direct consequence of Steps 1, 2 and 3.
	\end{proof}
	
	\begin{remark}
		Using the same ideas one can also verify  the sign identity below
		\begin{equation}\label{signidentityauxiliary}
			{\rm sign}(v^{(1)})={\rm sign}(\mathcal{E}_1(v)-v^{(5)})={\rm sign}(\Psi_{1,\alpha_1}(v)-v^{(5)})={\rm sign}(-K_4v^{(3)}+\alpha_1 v^{(1)}),
		\end{equation}
		where $\alpha_1\in(0,\frac{K_2}{2}]$ is such that \eqref{ordering8} holds.
	\end{remark}

	As a consequence of the last lemma, we find two energy inequalities for solutions to \eqref{ourODE}.
	\begin{corollary}\label{cor:energyidentity}
		Let $v\in C^{6}(\mathbb R)$ be a bounded solution to \eqref{ourODE}. The following energy identity holds:
		\begin{equation}\label{signfunction}
			-\mathcal{R}(v):=-\mathcal{H}(v)+G(v)+\frac{1}{2}{v^{(3)}}^2\leqslant0.
		\end{equation}
	\end{corollary}
	
	\begin{proof}
		Initially, notice that using \eqref{energyfactorized}, we get
		\begin{equation}\label{signfunction2}
			-\mathcal{R}(v):=-[\mathcal{E}_1(v)v^{(1)}+\mathcal{E}_2(v)v^{(2)}],
		\end{equation}
		which combined with Lemma~\ref{lm:signproperty} implies \eqref{signfunction}, and so the proof is concluded.
	\end{proof}
	
	The next comparison lemma is a preparation for proving the energy ordering.
	Subsequently, we will focus on the situation $v_{1}^{(1)}(0)>v_{2}^{(1)}(0) \geqslant 0$.
	By symmetry, the same strategy applies otherwise.
	This result is a weaker version of Lemma~\ref{lm:comparisonprinciple}, and strongly   relies on the energy identity in Corollary~\ref{cor:energyidentity}.
	It shows that to prove uniqueness of solutions to \eqref{ourODE}, one only needs three initial conditions instead of six as usual. 
	
	\begin{lemma}\label{lm:ordering-comparison}
		Let $v_1, v_2 \in C^{6}(\mathbb{R})$ be bounded solutions to \eqref{ourODE}.
		Suppose that $\mathcal{H}(v_{1}) \leqslant \mathcal{H}(v_{2})$, and
		\begin{align}\label{intitialconditionsordering}
			v_{1}(0)=v_{2}&(0), \quad v_{1}^{(1)}(0)>v_{2}^{(1)}(0) > 0, \quad v_{1}^{(2)}(0)=v_{2}^{(2)}(0), \quad {\rm and} \quad v_{1}^{(3)}(0)= v_{2}^{(3)}(0).
		\end{align}
		Then, $v_{1} \equiv v_{2}$.
	\end{lemma}
	
	\begin{proof}
		Initially, using \eqref{intitialconditionsordering}, it is not hard to check that
		\begin{equation}\label{ordering5}
			L_{\lambda_1}(v_1(0))\geqslant L_{\lambda_1}(v_2(0)) \quad {\rm and} \quad L_{\lambda_1}^{(1)}(v_1(0))\geqslant L_{\lambda_1}^{(1)}(v_2(0)).
		\end{equation}
		
		Next, from \eqref{ordering8} and \eqref{signfunction2} at $t=0$, we have
		\begin{equation*}
			-\mathcal{R}(v_1(0))={-\mathcal{H}(v_1(0))+G(v_1(0))+\frac{1}{2}{v_1^{(3)}}(0)^2}
		\end{equation*}
		and
		\begin{equation*}
			-\mathcal{R}(v_2(0))=-\mathcal{H}(v_2(0))+G(v_2(0))+\frac{1}{2}{v_2^{(3)}(0)}^2.
		\end{equation*}
		Also, using \eqref{intitialconditionsordering}, it follows 
		\begin{equation*}
			-\mathcal{H}(v_1(0))+G(v_1(0))+\frac{1}{2}{v_1^{(3)}(0)}^2\leqslant-\mathcal{H}(v_2(0))+G(v_2(0))+\frac{1}{2}{v_2^{(3)}(0)}^2\leqslant0.
		\end{equation*}
		This combined with  Corollary~\ref{cor:energyidentity} gives us
		\begin{equation}\label{ordering9}
			-\mathcal{R}(v_1(0)) \leqslant-\mathcal{R}(v_2(0))\leqslant 0.
		\end{equation}
		From this, we obtain two identities as follows
		\begin{equation*}
			0\leqslant\left[\mathcal{E}_1(v_2)v^{(1)}_2-\mathcal{E}_1(v_1)v_1^{(1)}\right](0)\leqslant\left[\mathcal{E}_1(v_2(0))-\mathcal{E}_1(v_1(0))\right]v_2^{(1)}(0)
		\end{equation*}
		and
		\begin{equation*}
			0\leqslant\left[\mathcal{E}_2(v_2)v_2^{(2)}-\mathcal{E}_2(v_1)v_1^{(2)}\right](0)\leqslant\left[\mathcal{E}_2(v_2(0))-\mathcal{E}_2(v_1(0))\right]v_2^{(2)}(0),
		\end{equation*}
		which yields
		\begin{equation}\label{ordering11}
			\mathcal{E}_1(v_2(0))\geqslant\mathcal{E}_1(v_1(0)) \quad {\rm and} \quad  \mathcal{E}_2(v_2(0))\geqslant\mathcal{E}_2(v_1(0)).
		\end{equation}
		
		Now, by combining \eqref{ordering11} and \eqref{intitialconditionsordering}, we get respectively  
		\begin{equation*}
			v_1^{(5)}(0)\geqslant v_2^{(5)}(0) \quad {\rm and} \quad v_1^{(4)}(0)\leqslant v_2^{(4)}(0).
		\end{equation*}
		Therefore, we conclude
		\begin{equation}\label{ordering6}
			(L_{\lambda_1}\circ L_{\lambda_2})(v_1(0))\geqslant (L_{\lambda_1}\circ L_{\lambda_2})(v_2(0)).
		\end{equation}
		and
		\begin{equation}\label{ordering10}
			(L_{\lambda_1}\circ L_{\lambda_2})^{(1)}(v_1(0))\geqslant (L_{\lambda_1}\circ L_{\lambda_2})^{(1)}(v_2(0)).
		\end{equation}
		Finally, using \eqref{intitialconditionsordering}, \eqref{ordering5}, \eqref{ordering6} and \eqref{ordering10}, we can apply Lemma~\ref{lm:weakcomparisonprinciple} to finish the proof.
	\end{proof}
	
	In what follows, let $[t_1, t_2]\subset\mathbb R$ denote intervals where the functions $v^{(1)},v^{(2)}$ are strictly monotone on $(t_1,t_2)$.
	In this fashion, denoting the inverse of $v(t)$ by $t(v)$, we define the change of variables
	\begin{equation}\label{changeofvariablesvdb}
		s=v \quad \text{ and } \quad \zeta(s)=\left[v^{(1)}(t(s))\right]^{2} .
	\end{equation}
	Now for $s \in[s_1, s_2]=[v(t_{1}), v(t_{2})]$, we have
	\begin{equation*}
		\zeta^{(1)}(s)=2 v^{(2)}(t(s)) \quad {\rm and} \quad \zeta^{(2)}(s)=2 v^{(3)}(t(s)).
	\end{equation*}
	If $t_{1}=-\infty$, then we write $\zeta^{(1)}(s_{1})=\lim_{s \rightarrow s_1}\zeta^{(1)}(t)$ and $\zeta^{(2)}(s_{1})=\lim_{s \rightarrow s_1}\zeta^{(2)}(t)$.
	In the same way, if $t_{2}=+\infty$, then we write $\zeta^{(1)}(s_{2})=\lim_{s \rightarrow s_2}\zeta^{(1)}(t)$ and $\zeta^{(2)}(s_{2})=\lim_{s \rightarrow s_2}\zeta^{(2)}(t)$.
	Notice that these two limits exist as a consequence of Lemma~\ref{lm:energylimits}.
	
	We are ready prove the energy ordering lemma, which is another crucial result in this manuscript.
	
	\begin{lemma}\label{lm:energyordering}
		Let $v_1, v_2 \in C^{6}(\mathbb{R})$ be bounded solutions to \eqref{ourODE}.
		Suppose that $v_1(0)=v_2(0)$ and either $v_1^{(1)}(0)>v_2^{(1)}(0) \geqslant 0$ or $v_1^{(1)}(0)<v_2^{(1)}(0) \leqslant 0$, then $\mathcal{H}(v_1)>\mathcal{H}(v_2)$.
	\end{lemma}
	
	\begin{proof}
		We divide the proof into three steps as follows.
		
		The first step states that there are at most two bounded solutions on the unstable manifold of each equilibrium point.
		
		\noindent{\bf Step 1.} Suppose that $v_{1}, v_{2},v^{(1)}_{1}, v^{(1)}_{2}$ are non-constant functions, and that there exists $\ell_0 \in \mathcal{A}(v_1)\cap\mathcal{A}(v_2)$ such that 
		\begin{equation*}
			\lim_{t \rightarrow\pm\infty} v_{1}(t)=\lim_{t \rightarrow\pm\infty} v_{2}(t)=\ell_0 \quad {\rm and} \quad \lim_{t \rightarrow\pm\infty} v^{(1)}_{1}(t)=\lim_{t \rightarrow\pm\infty} v^{(1)}_{2}(t)=0.
		\end{equation*}
		Then, it follows that $v_{1}(t)$ decreases to $\ell_0$ and $v_{2}(t)$ increases to $\ell_0$ as $t\rightarrow\pm\infty$, or vice versa.
		In addition, $v^{(1)}_{1}(t)$ increases to $0$ and $v^{(1)}_{2}(t)$ decreases to $0$ as $t\rightarrow\pm\infty$, or vice versa.
		
		Initially, by Corollary~\ref{cor:symmetry}, $v_{1}, v_{2}$ and $v^{(1)}_{1}, v^{(1)}_{2}$ can only tend monotonically to their respective limits. 
		Also, using symmetry, we may suppose by contradiction that $v_{1}$ and $v_{2}$ both decrease towards $\ell_0$ as $t\rightarrow-\infty$ and $v^{(1)}_{1}$ and $v^{(1)}_{2}$ both increase towards $0$ as $t\rightarrow-\infty$.
		In other words, there exists $T\gg1$ such that $v_{1}(t)>0$, $v_{2}(t)>0$, $v_{1}^{(1)}(t)<0$, and $v_{2}^{(1)}(t)<0$ for $t \in(-\infty, T)$.
		
		We finish the contradiction argument with the following claim:
		
		\noindent{\bf Claim 1.} $v_{1} \equiv v_{2}$. 
		
		\noindent Indeed, for $s \in(\ell_0, \ell_0+\varepsilon_{0})$, where $0<\varepsilon_{0}\ll1$ is small enough, let $\zeta_{1}$ and $\zeta_{2}$ correspond to $v_{1}$ and $v_{2}$ respectively by the change of variables defined in \eqref{changeofvariablesvdb}. 
		Notice that one must have $\zeta^{(1)}_{1} \neq \zeta^{(1)}_{2}$ and $\zeta^{(2)}_{1} \neq \zeta^{(2)}_{2}$ on $(\ell_0, \ell_0+\varepsilon_{0})$ .
		Otherwise, by Lemma~\ref{lm:comparisonprinciple}, we would get $v_{1} \equiv v_{2}$. 
		Without loss of generality, we may assume $\zeta^{(1)}_{1}>\zeta^{(1)}_{2}$ and $\zeta^{(2)}_{1}>\zeta^{(2)}_{2}$ on $(\ell_0, \ell_0+\varepsilon_{0})$. 
		Since both $\zeta_{1},\zeta^{(1)}_{1}$ and $\zeta_{2},\zeta^{(1)}_{2}$ are differentiable on $(\ell_0, \ell_0+\varepsilon_{0})$ and, by Lemma~\ref{lm:energylimits},
		$\zeta_{1}(\ell_0)=\zeta_{2}(\ell_0)=0$ and $\zeta_{1}^{(1)}(\ell_0)=\zeta_{2}^{(1)}(\ell_0)=0$.
		Therefore, there exists $s_{0} \in(\ell_0, \ell_0+\varepsilon_{0})$ such that $\zeta^{(1)}_{1}(s_{0}) \geqslant \zeta^{(1)}_{2}(s_{0})$ and $\zeta^{(2)}_{1}(s_{0})\geqslant \zeta^{(2)}_{2}(s_{0})$.
		
		We divide the analysis into two cases:
		
		\noindent{\bf Case 1.} Both $\zeta_{2}^{(1)}(s_0) \geqslant 0$ and $\zeta_{2}^{(2)}(s_0) \geqslant 0$.
		
		\noindent There exist $t_{1},t_{2}\in\mathbb{R}$ such that
		\begin{align*}
			v_{1}(t_1)=v_{2}(t_2)&=s_0, \quad v_{1}^{(1)}(t_1)>v_{2}^{(1)}(t_2) > 0, \quad v_{1}^{(2)}(t_1)\geqslant v_{2}^{(2)}(t_2), \quad \text {and} \quad v_{1}^{(3)}(t_1)\geqslant v_{2}^{(3)}(t_2).
		\end{align*}
		By translating, $v_{1}$ and $v_{2}$ by $t_{1}$ and $t_{2}$ respectively, we find
		\begin{align}\label{ordering1}
			v_{1}(0)=&v_{2}(0), \quad v_{1}^{(1)}(0)>v_{2}^{(1)}(0) > 0, \quad v_{1}^{(2)}(0)\geqslant v_{2}^{(2)}(0), \quad \text {and} \quad v_{1}^{(3)}(0)\geqslant v_{2}^{(3)}(0).
		\end{align}
		Since both $v_{1}$ and $v_{2}$ tend to $\ell_0$ monotonically as $t\rightarrow-\infty$, we conclude from Lemma~\ref{lm:energylimits} that       $\mathcal{A}^-(v_1)=\mathcal{A}^-(v_2)=\ell_0$ and $\mathcal{A}^-(v^{(j)}_1)=\mathcal{A}^-(v^{(j)}_2)=0$ for $j=1,2,3,4,5$. 
		Therefore, $\mathcal{H}(v_{1})=\mathcal{H}(v_{2})$. 
		In the light of \eqref{signidentityauxiliary}, it is straightforward to check that \eqref{ordering1} instead of \eqref{intitialconditionsordering} is sufficient for the proof in Lemma~\ref{lm:ordering-comparison} to go on unchanged. 
		Hence $v_{1} \equiv v_{2}$. 
		This finishes the proof of the first case.
		
		\noindent{\bf Case 2.} Either $\zeta_{2}^{(1)}(s_0) < 0$ or $\zeta_{2}^{(2)}(s_0) < 0$.
		
		\noindent Without loss of generality, we assume $\zeta_{2}^{(1)}(s_0) < 0$; the same idea applies in the other situation.
		Notice that $\zeta^{(1)}_{2}>0$ on $(\ell_0, s_{0}]$ and $\zeta^{(1)}_{2}(\ell_0)=0$.
		Then, there exists a $\bar{s}_{0} \in(\ell_0, s_{0})$ such that $\zeta_{2}^{(1)}(\bar{s}_{0})=0$. 
		If $\zeta_{1}^{(1)}(\bar{s}_{0}) \geqslant 0$, then $\zeta_{1}^{(1)}(\bar{s}_{0}) \geqslant \zeta_{2}^{(1)}(\bar{s}_{0})=0$, which is covered by the last case. 
		If $\zeta_{1}^{(1)}(\bar{s}_{0})<0$, then $ \zeta_{1}^{(1)}(\bar{s}_{0})<\zeta_{2}^{(1)}(\bar{s}_{0})$ and $\zeta_{1}^{(1)}(s_0) \geqslant \zeta_{2}^{(1)}(s_{0})$, and by continuity there exists $\tilde{s}_{0} \in(\bar{s}_{0}, s_{0}]$ such that $\zeta_{1}^{(1)}(\tilde{s}_{0})=\zeta_{2}^{(1)}(\tilde{s}_{0})$. 
		Hence, one can find $t_{1},t_2\in\mathbb{R}$ such that
		\begin{align*}
			v_{1}(t_1)=v_{2}(&t_2)=\tilde{s}_{0}, \quad v_{1}^{(1)}(t_1)>v_{2}^{(1)}(t_2) > 0, \quad v_{1}^{(2)}(t_1)= v_{2}^{(2)}(t_2), \quad \text {and} \quad v_{1}^{(3)}(t_1)\geqslant v_{2}^{(3)}(t_2).
		\end{align*}
		Since $\mathcal{H}(v_1)=\mathcal{H}(v_{2})$ as above, we may apply Lemma~\ref{lm:ordering-comparison} to obtain $v_{1} \equiv v_{2}$, which finishes the proof of the second case.
		
		The proof of Step 1 is concluded.
		
		The next step states that whenever the energy inequality is violated one can find points on which the second and the third derivatives  of these two bounded solutions coincide. 
		
		\noindent{\bf Step 2.} Suppose that $v_{1}(0)=v_{2}(0)$ and $v_{1}^{(1)}(0)>v_{2}^{(1)}(0) \geqslant 0$, and $\mathcal{H}(v_1) \leqslant \mathcal{H}(v_2)$. 
		Then, there exist $t_{1},t_{2}\in\mathbb{R}$ such that 
		\begin{align}\label{ordering0}
			v_{1}(t_{1})=&v_{2}(t_{2}), \quad v_{1}^{(1)}(t_{1})>v_{2}^{(1)}(t_{2}) \geqslant 0, \quad v_{1}^{(2)}(t_{1})=v_{2}^{(2)}(t_{2}), \quad \text{and} \quad
			v_{1}^{(3)}(t_{1})= v_{2}^{(3)}(t_{2}).
		\end{align}
		
		Let $[\hat{t}_{1}, \hat{t}_{2}]$ be the largest interval containing $t=0$ on which $v_{1}^{(1)},v_{1}^{(2)},v_{1}^{(3)}$ are positive, and let $[t_{1}, t_{2}]$ be the largest interval containing $t=0$ on which $v_{2}^{(1)},v_{2}^{(2)},v_{2}^{(3)}$ are positive.
		We can now apply the change variables \eqref{changeofvariablesvdb} again. 
		Let $\zeta_{1}$ correspond to $v_{1}$ on $[\hat s_{1}, \hat s_{2}]=[v_{1}(\hat{t}_{1}), v_{1}(\hat{t}_{2})]$, and let $\zeta_{2}$ correspond to $v_{2}$ on $[s_{1}, s_{2}]=[v_{2}(t_{1}), v_{2}(t_{2})]$. 
		Clearly, as a consequence of Lemma~\ref{lm:comparisonprinciple}, we have $\zeta_{1}>\zeta_{2}$ on $(s_{1}, s_{2})$.
		Finally, if $-\infty<t_{1}$, then it follows from Lemma~\ref{lm:energylimits} that $\hat{s}_{1}<s_{1}$, whereas if $t_1=-\infty$, then this follows from Step 1. 
		The same happens to $\hat{s}_{2}>s_{2}$.
		
		We fix the notation $\boldsymbol{\zeta}=(\zeta^{(1)},\zeta^{(2)})$ and $|\boldsymbol{\zeta}|^2={\zeta^{(1)}}^2+{\zeta^{(2)}}^2$.
		We will prove that the norm function has opposite sign on the edges of $[s_1,s_2]$.
		
		\noindent{\bf Claim 2.} $|\boldsymbol{\zeta}_{1}(s_1)|^2<|\boldsymbol{\zeta}_{2}(s_{1})|^2$.
		
		\noindent One can see that $\zeta_{2}(s_1)=v_{2}^{(1)}(t_1)=0$, $\zeta_{2}^{(1)}(s_1)=2 v_{2}^{(2)}(t_1) = 0$ and $\zeta_{2}^{(2)}(s_1)=2 v_{2}^{(3)}(t_1) \geqslant 0$.
		Indeed, if $t_1=-\infty$, this follows from Lemma~\ref{lm:energylimits}, whereas if $-\infty<t_1$, then it holds because $v(t_1)$ and $v^{(1)}(t_1)$ are minimums. 
		Notice that $v_{1}(\hat{t}_{1})=\hat{s}_{1}<s_1$ and $v_{1}(\hat{t}_{2})=\hat{s}_{2}>s_2>s_1$, thus let $t_*\in(\hat{t}_{1}, \hat{t}_{2})$ be such that $v_{1}(t_*)=v_{2}(t_1)=s_{1}$. 
		By \eqref{energyfactorized} and Lemma~\ref{lm:signproperty}, we obtain
		\begin{align*}
			\frac{|\boldsymbol{\zeta}_{2}(s_1)|^2-|\boldsymbol{\zeta}_{1}(s_{1})|^2}{8}
			&=\left(\frac{1}{2}v_{2}^{(2)}(t_1)^{2}-\frac{1}{2}v_{1}^{(2)}(t_*)^{2}\right)+\left(\frac{1}{2}v_{2}^{(3)}(t_1)^{2}-\frac{1}{2}v_{1}^{(3)}(t_*)^{2}\right) \\
			&\geqslant\left[\mathcal{H}(v_2)-G(s_1)\right]-\left[\mathcal{H}(v_1)-G(s_1)-\mathcal{R}(v_1(t_*))\right] \\
			&=[\mathcal{H}(v_2)-\mathcal{H}(v_1)]+\mathcal{R}(v_1(t_*)).
		\end{align*}
		Next, using that  $v_{1}^{(2)}(t_*)=\sqrt{\zeta^{(1)}_{1}(s_1)}>\sqrt{\zeta^{(1)}_{2}(s_1)}=0$, combined with Corollary~\ref{cor:energyidentity}, we conclude
		$\mathcal{R}(v_1(t_*))>0$.
		Hence, since $\mathcal{H}(v_2)- \mathcal{H}(v_1)\geqslant0$, the desired conclusion holds.
		
		\noindent{\bf Claim 3.} $|\boldsymbol{\zeta}_{1}(s_2)|^2>|\boldsymbol{\zeta}_{2}(s_{2})|^2$. 
		
		\noindent Indeed, one can see that $\zeta_{2}(s_2)=0$ and $\zeta_{2}^{(1)}(s_2)=2 v_{2}^{(2)}(t_1) \leqslant 0$.
		Indeed, if $t_2=\infty$, this follows from Lemma~\ref{lm:energylimits}, whereas if $t_2<+\infty$, then it holds because $v(t_2)$ is a maximum. 
		Notice that $v_{2}({t}_{1})={s}_{1}>\hat{s}_1$ and $v_{2}({t}_{2})={s}_{2}<\hat{s}_2<\hat{s}_1$, thus let $t^*\in({t}_{1}, {t}_{2})$ be such that $v_{2}(t^*)=v_{1}(\hat{t}_2)=\hat{s}_{2}$. 
		By \eqref{energyfactorized} and Lemma~\ref{lm:signproperty}, we obtain
		\begin{align*}
			\frac{|\boldsymbol{\zeta}_{2}(s_2)|^2-|\boldsymbol{\zeta}_{1}(s_{2})|^2}{8}
			=[\mathcal{H}(v_2)-\mathcal{H}(v_1)]+\mathcal{R}(v_1(t_2)).
		\end{align*}
		As before, we conclude the proof of this claim.
		
		By continuity, we find $s_{*} \in(s_{1}, s_{2})$ such that $|\boldsymbol{\zeta}_{1}(s_*)|=|\boldsymbol{\zeta}_{2}(s_{*})|$, which implies  ${\zeta}^{(1)}_{1}(s_*)={\zeta}^{(1)}_{2}(s_{*})$ and ${\zeta}^{(2)}_{1}(s_*)={\zeta}^{(2)}_{2}(s_{*})$.
		The proof is finished by undoing the transformation \eqref{changeofvariablesvdb}
		
		Finally, we are based on the last step to conclude the proof of the energy ordering.
		
		\noindent{\bf Step 3.} $\mathcal{H}(v_1)>\mathcal{H}(v_2)$.
		
		By symmetry, we may assume $v_{1}^{(1)}(0)>v_{2}^{(1)}(0) \geqslant 0$.
		Suppose by contradiction that  $\mathcal{H}(v_1)\leqslant\mathcal{H}(v_2)$, then, by Step 2, there exist $t_{1},t_{2}\in\mathbb R$ such that \eqref{ordering0} holds.
		Thus, by translation invariance, we may take $t_{1}=t_{2}=0$, which gives us that \eqref{intitialconditionsordering} are satisfied.
		Therefore, using Lemma~\ref{lm:ordering-comparison} we obtain that $v_{1} \equiv v_{2}$, which contradicts our assumptions, and so Step 3 is proved.
		
		The proof of the lemma is finished.
	\end{proof}
	
	\subsection{Uniqueness}
	We prove uniqueness, up to translations, of homoclinic solutions to \eqref{ourODE}. 
	\begin{lemma}\label{lm:uniqueness}
		Let $v_1, v_2 \in C^{6}(\mathbb{R})$ be bounded positive solutions to \eqref{ourODE}.
		Suppose that $\lim_{|t| \rightarrow \infty} v_1(t)=\lim _{|t| \rightarrow \infty}v_2(t)=0$ and $v_1^{(1)}(0)=v_2^{(1)}(0)=0$.
		Then, it follows $v_1 \equiv v_2$.
	\end{lemma}
	
	\begin{proof}
		We divide the proof into claims
		
		\noindent{\bf Claim 1.} $t_0=0$ is the only zero of $v_1^{(1)}$ and $v_2^{(1)}$. 
		
		\noindent Indeed, if $v_1^{(1)}$ had another zero at, say, $t_{1}>0$, then by repeatedly applying Corollary~\ref{cor:symmetry}, we get that $v$ must be periodic of period $2t_{1}>0$.
		In particular $0<v(0)=v(2kt_{1})$ for all $k \in \mathbb{N}$, which contradicts the fact $v(t) \rightarrow 0$ as $t \rightarrow +\infty$. 
		The argument for $v_2^{(1)}$ is analogous. 
		Hence, we get
		\begin{equation}\label{uniqueness1}
			v_1^{(1)}(t)<0 \quad \text{and} \quad v_2^{(1)}(t)<0 \quad \text{for all} \quad t>0;
		\end{equation}
		this proves the first claim.
		
		Next, by Lemma~\ref{lm:energylimits} and Proposition~\ref{prop:conservation}, we find
		\begin{equation}\label{uniqueness2}
			\mathcal{H}_{v_1}=\lim _{t \rightarrow +\infty} \mathcal{H}_{v_1}(t)=G(0)=0 \quad \text{and} \quad \mathcal{H}_{v_2}=\lim _{t \rightarrow +\infty} \mathcal{H}_{v_2}(t)=G(0)=0.
		\end{equation}
		From which, we see that if $v_1(0)=v_2(0)$, then, by Lemma~\ref{lm:comparisonprinciple}, the proof is concluded.
		
		\noindent{\bf Claim 2.} $v_1(0)=v_2(0)$.
		
		\noindent As a matter of fact, let us suppose by contradiction that $v_1(0)>v_2(0)$. 
		We claim that this implies $v_1>v_2$ everywhere. 
		Indeed, otherwise there would exist $t_{1}>0$ such that $v_1>v_2$ on $[0, t_{1})$ and $v_1(t_{1})=v_2(t_{1})$.
		Then, by \eqref{uniqueness1}, we infer that $v_1^{(1)}(t_{1}) \leqslant v_2^{(1)}(t_{1})<0$.
		If $v_1^{(1)}(t_{1})=v^{(1)}_2(t_{0})$, 
		then Lemma~\ref{lm:comparisonprinciple} implies $v_1 \equiv v_2$, contradicting $v_1(0)>v_2(0)$.
		If $v_1^{(1)}(t_{1})<v_2^{(1)}(t_{1})<0$, then, by Lemma~\ref{lm:energyordering}, we obtain $\mathcal{H}_{v_1}>\mathcal{H}_{v_2}$, which contradicts \eqref{uniqueness2}.
		
		For any $R>0$, since $v_1,v_2\in C^{6}(\mathbb R)$ satisfy \eqref{ourODE}, using integration by parts, we get
		\begin{align*}
			0&=\int_{-T}^{T} v_2\left(v_1^{(6)}-K_4v_1^{(4)}+K_2v_1^{(2)}+g(v_1)\right) \\
			&=B(T)+\int_{-T}^{T} v_1\left(v_2^{(6)}-K_4v_2^{(4)}+K_2v_2^{(2)}+g(v_2)\right)+\int_{-T}^{T} v_2v_1\left(g(v_2)-g(v_1)\right) \\
			&=B(T)+\int_{-T}^{T} v_2 v_1\left(g(v_2)-g(v_1)\right),
		\end{align*}
		where the term $B(T)$ denotes all the boundary terms coming from the integrations by parts. 
		By Lemma~\ref{lm:energylimits}, we have $B(T) \rightarrow 0$ as $R \rightarrow +\infty$. 
		However, since $\int_{-T}^{T} v_2v_1(f(v_2)-f(v_1))$ is a negative and strictly decreasing function of $T$, we obtain a contradiction by choosing $T\gg1$ large enough.
		
		The lemma is proved.
	\end{proof}
	
	A direct consequence of this uniqueness result is that for the concrete values of $K_0,K_2,K_4$ given by \eqref{coefficients} and $p=\frac{2n}{n-6}$, one can compute the homoclinic solution explicitly. 
	This is precisely the only place in the proof where the closed form of these parameters enters.
	\begin{corollary}\label{cor:closedform}
		Let $v\in C^{6}(\mathbb{R})$ be a bounded solution to \eqref{ourODE}.
		If $\lim_{|t|\rightarrow\infty}v(t)=0$, then there exists $T \in \mathbb{R}$ such that
		\begin{equation*}
			v(t)=c_{n}(2\cosh(t-T))^{-\gamma_n}.
		\end{equation*}
	\end{corollary}
	
	\begin{proof}
		A straightforward calculation shows that $v_0(t)=c_{n}(2 \cosh (t))^{-\gamma_n}$ solves \eqref{ourODE}. 
		Using the conditions on $v$, we have that $v$ has a global maximum at some $T \in \mathbb{R}$. 
		Since $v^{(1)}(T)=0$, we can apply Lemma~\ref{lm:uniqueness} to conclude that $v(\cdot+T)=v_0$, and this finishes the proof.
	\end{proof}
	
	\subsection{Classification}
	In what follows, we state the most important auxiliary result in the proof of our main proposition.
	\begin{lemma}\label{lm:classification}
		Let $v \in C^{6}(\mathbb{R})$ be a solution to \eqref{ourODE}.
		One of the following three alternatives holds:
		\begin{itemize}
			\item[{\rm (a)}] $v\equiv \pm a^{*}_n$ or $v \equiv 0$;
			\item[{\rm (b)}] $v(t)=\pm c_{n}(2\cosh (t-T))^{-\gamma_n}$ for some $T \in \mathbb{R}$;
			\item[{\rm (c)}] $v$ is periodic, has a unique local maximum and minimum per period, and is symmetric with respect to its local extrema.
		\end{itemize}
	\end{lemma}
	
	\begin{proof}
		Let $v \in C^{6}(\mathbb{R})$ be a solution to \eqref{ourODE} and let $\#Z^{(1)}({v})$ be its critical set as defined previously in \eqref{criticalset}.
		We now distinguish several cases with respect to the cardinality of the critical set.
		
		\noindent{\bf Case 1.} $\#Z^{(1)}({v})=0$.
		
		\noindent In this case, $v$ is strictly monotone. 
		In fact, up to replacing $v(t)$ by $v(-t)$, we may assume that $v$ is strictly increasing, and so, by Lemma~\ref{lm:asymptotics}, it follows $\ell_{\pm}=\lim_{t \rightarrow \pm \infty} v(t)\in\mathbb{R}$. 
		By Lemma~\ref{lm:energylimits}, we are reduced to studying three situations as follows.
		
		If $\ell_{-}=0$ and $\ell_{+}=a^{*}_n$, then using Lemma~\ref{lm:energylimits}, we get $\lim_{t \rightarrow-\infty} \mathcal{H}_{v}(t)=G(0)=0$, whereas $\lim_{t \rightarrow+\infty} \mathcal{H}_{v}(t)=G(a^{*}_n)<0$, a contradiction to energy conservation. In the same way, a contradiction is obtained if $\ell_{-}=-a^{*}_n$ and $\ell_{+}=0$. 
		
		It remains to consider the case $\ell_{-}=-a^{*}_n$ and $\ell_{+}=a^{*}_n$. 
		As before, we can apply Lemma~\ref{lm:energylimits}
		\begin{equation}\label{classification2}
			\lim_{|t| \rightarrow \infty} \mathcal{H}_{v}(t)=G(a^{*}_n)<0.
		\end{equation}
		On the other hand, by Corollary~\ref{cor:energyidentity}, we have that the following sign identity holds
		\begin{equation}\label{classification1}
			\mathcal{H}_{v}(t) \geqslant \frac{1}{2} v^{(3)}(t)^{2}+\mathcal{R}(v(t))+G(v(t)).
		\end{equation}
		However, by evaluating the energy at some $t=t_0$ such that $v(t_0)=0$ and using \eqref{classification1}, we obtain $\mathcal{H}_{v}(t_{0}) \geqslant \frac{1}{2} v^{(3)}(t_0)^{2}+\mathcal{R}(0)+G(0) \geqslant 0$,
		which contradicts \eqref{classification2} and the energy conservation. 
		The conclusion is that $Z^{(1)}({v})=\varnothing$ cannot occur.
		
		\noindent{\bf Case 2.} $\# Z^{(1)}({v})=1$.
		
		\noindent Up to a translation, we may assume $Z^{(1)}({v})=\{0\}$. 
		Thus, $v$ is strictly monotone on $(-\infty, 0)$ and $(0, \infty)$, which yields that $\ell_{\pm}=\lim_{t \rightarrow \pm \infty} v(t)\in\mathbb{R} \cup\{\pm \infty\}$. 
		By Lemma~\ref{lm:asymptotics}, we have $\ell_{\pm}\in\mathbb R$, and so $v$ is bounded, and, by Corollary~\ref{cor:symmetry}, also even. Then, we find $\ell_{+}=\ell_{-}$. 
		By Lemma~\ref{lm:energylimits}, only three cases can occur: $\ell_{+}=\ell_{-}=0$ or $\ell_{+}=\ell_{-}=\pm a^{*}_n$. 
		In this case, monotonicity implies that either $v>0$ or $v<0$, and we conclude $v(t)=\pm c_{n}(2\cosh(t))^{-\gamma_n}$ by Corollary~\ref{cor:closedform}.
		
		For the remaining cases, by otherwise replacing $v(t)$ by $-v(t)$, let us assume without loss of generality that $\ell_{+}=\ell_{-}=a_*$. 
		Using that $v$ is strictly monotone on $[0, \infty)$, $v(0) \neq a^{*}_n$, and $\mathcal{H}_{v}(0) \geqslant \frac{1}{2} v^{(3)}(0)^{2}+G(0) \geqslant 0$,         one can see that since $F$ attains its global minimal value only at $v=\pm a^{*}_n$ it holds $v(0)=-a^{*}_n$. Hence $v$ changes sign; {\it i.e.}, there exists $t_{0} \in \mathbb{R}$ such that $v\left(t_{0}\right)=0$. By Corollary~\ref{cor:symmetry}, we get that $v$ is anti-symmetric with respect to $t_{0}$, which is a contradiction with  $\ell_{\pm}>0$. Therefore, we obtained that if $\# Z^{(1)}({v})=1$, then $v(t)=\pm c_{n}(2\cosh(t))^{-\gamma_n}$.
		
		\noindent{\bf Case 3.} $\# Z^{(1)}(v)\geqslant2$.
		
		\noindent By continuity of $v^{(1)}$, we find that unless $v$ is constant (and hence $v \equiv \pm a^{*}_n$ or $v \equiv 0)$, the closed set $Z^{(1)}(v)$ cannot be dense; that is, there exist real numbers $t_1<t_2$ such that $v^{(1)}(t_1)=v^{(1)}(t_2)=0$ and $v^{(1)} \neq 0$ on $(t_1, t_2)$. 
		Then, by Lemma~\ref{lm:boundededness}, we get that $v$ is bounded and therefore, we can use Corollary~\ref{cor:symmetry} as in the proof of Lemma~\ref{lm:uniqueness} to find that $v$ must be periodic of period $2(t_2-t_1)$. 
		In addition, since $v$ is strictly monotone on $(t_1, t_2)$, there is (strictly) only one maximum and minimum per period interval. The symmetry with respect to the extrema follows from Corollary~\ref{cor:symmetry}.
	\end{proof}
	
	\subsection{Shooting technique}
	Finally, we are ready to prove the main result in this section.
	\begin{proof}[Proof of Proposition~\ref{prop:odeclassification}]
		The proof will be divided into three steps as follows:
		
		\noindent{\bf Step 1.} Proof of \ref{itm:C1}.
		
		By Lemma~\ref{lm:classification}, the only possible case on which the estimate \eqref{estimate} may fail to is when $v$ is periodic. 
		In this case, $v$ possesses a local minimum at, say, $t_{0} \in \mathbb{R}$. 
		Note that if $v$ has a zero then \eqref{estimate} is automatically fulfilled, so we may assume that $v$ has a fixed sign and, up to replacing $v$ by $-v$, we may assume that $v>0$. 
		However, using Lemma~\ref{lm:maximumandminimumpoints}, either $v$ is constant, and so $v \equiv a^{*}_n$, or $v(t_{0})<a^{*}_n$, which yields that \eqref{estimate} holds with strict inequality, which in turn proves (i).
		
		\noindent{\bf Step 2.} Proof of \ref{itm:C2}.
		
		The value $a_0 \in(0, a^{*}_n)$ will be considered to be fixed throughout the following argument.
		For any vector $\boldsymbol{b}=(b_1,b_2) \in\mathbb R^{2}_+:=\{\boldsymbol{b} : b_j\geqslant 0 \; {\rm for} \; j=1,2\}$, where $b_j=a_{2j}$ for $j=1,2$, let us denote by $v_{\boldsymbol{b}}\in C^{6}(\mathbb R)$ the unique solution to \eqref{ourODE} with the initial values
		\begin{equation}\label{shootinginitialvalues}
			v(0)=a_0, \quad v^{(2)}(0)=b_1, \quad v^{(4)}(0)=b_2, \quad {\rm and} \quad v^{(1)}(0)=v^{(3)}(0)=v^{(5)}(0)=0,
		\end{equation}
		and by $T_{\boldsymbol{b}} \in(0, \infty]$ its maximal forward time of existence.
		
		We now proceed to our shooting technique.
		Since $G(v) \rightarrow +\infty$ as $v \rightarrow +\infty$, for any $C(a_0)>0$, which will be chosen later, there exists $R(a_0)>0$ such that $G(v)>C(a_0)$ for all $v>R(a_0)$. 
		In this fashion, let us define the shooting decomposition sets
		\begin{equation*}
			S_1:=\left\{\boldsymbol{b}\in\mathbb R^{2}_+ : v_{\boldsymbol{b}}(t)<0 \ {\rm for \ some} \ t\in(0, T_{\boldsymbol{b}})\right\},
		\end{equation*}
		and
		\begin{equation*}
			S_2:=\left\{\boldsymbol{b}\in\mathbb R^{2}_+ : v_{\boldsymbol{b}}(t)>R(a_0)\ {\rm for \ some} \ t \in(0, T_{\boldsymbol{b}}) \ {\rm and} \ v_{\boldsymbol{b}}>0 \ {\rm in }\ [0, t]\right\}.
		\end{equation*}
		Clearly, $S_1,S_2\subset \mathbb R^{m-1}_+$ are open because of the standard continuous dependence of solutions on the initial conditions.
		Thus, we are left to show that they are both non-empty and disjoint.
		This is the content of the following claims.
		
		\noindent{\bf Claim 1:} $(b^*,+\infty)^{2}\subset S_2\neq\varnothing$, where $b^{*}=\max\{\beta_1,\beta_2\}$ are the greatest solutions to the algebraic equation $-K_{4}\beta^{4}+K_{2}\beta^{2}-\widehat{K}_0=0$ in $\mathbb R$.
		
		\noindent In fact, suppose that $b_j>b^{*}$ for any $j=1,2$.
		Hence, writing \eqref{ourODE} as
		\begin{equation}\label{odeclassification1}
			v_{\boldsymbol{b}}^{(6)}=K_{4}v_{\boldsymbol{b}}^{(4)}-K_{2}v_{\boldsymbol{b}}^{(2)}-g(v_{\boldsymbol{b}}) \quad {\rm in} \quad \mathbb R
		\end{equation}
		and combining with \eqref{shootinginitialvalues}, we find that $v_{\boldsymbol{b}}^{(6)}$ increases initially. 
		Thus, $v_{\boldsymbol{b}}^{(2)}(0)$ and $v_{\boldsymbol{b}}^{(4)}(0)$ increase initially, and since the right-hand side of \eqref{odeclassification1} is positive initially, it is easy to check that they stay positive on $[0, T_{\boldsymbol{b}})$. Whence, $v_{\boldsymbol{b}}^{(6)}>0$ in $[0, T_{\boldsymbol{b}})$, which implies that $v^{(j)}_{\boldsymbol{b}}$ for $j=0,\dots, 5$ all keep increasing on $[0, T_{\boldsymbol{b}})$. 
		Consequently, if $T_{\boldsymbol{b}}=+\infty$ then $v_{\boldsymbol{b}}$ is unbounded. 
		On the other hand, since $g$ is locally Lipschitz, if $T_{\boldsymbol{b}}<+\infty$, then $v_{\boldsymbol{b}}(t) \rightarrow +\infty$ as $t \rightarrow T_{\boldsymbol{b}}$ . To summarize, $v_{\boldsymbol{b}}^{(1)}>0$ in $[0, T_{\boldsymbol{b}})$ and $\lim_{t \rightarrow T_{\boldsymbol{b}}}v_{\boldsymbol{b}}(t)=+\infty$ when $b_j \geqslant b^*$ for any $j=1,2$.
		Based on this, we can restrict our search to the so-called shooting parameter set $\boldsymbol{b} \in[0, b^*]^{2}=:S$.
		
		\noindent{\bf Claim 2:} $(S_1\cap S_2)\cap S=\varnothing$.
		
		\noindent In fact, notice that for all $\boldsymbol{b} \in S$, we have the uniform energy bound
		\begin{equation}\label{energybound}
			\mathcal{H}({v_{\boldsymbol{b}}})(0)\leqslant -b_2b_1+b_1^2+G(a_0)\leqslant (b_1-b_2)b_1+G(a_0):=C(a_0).
		\end{equation}
		This implies that whenever $\boldsymbol{b} \in[0,b^*]^{2}$ and $v_{\boldsymbol{b}}(t_{0})>R(a_0)$, we must have $v_{\boldsymbol{b}}^{(1)}(t_{0}) \neq 0$, otherwise, by using Corollary~\ref{cor:energyidentity}, the following identity holds
		\begin{align*}
			\mathcal{H}({v_{\boldsymbol{b}}}(t_0))=\frac{1}{2}{{{v}_{\boldsymbol{b}}^{(3)}(t_0)}}^{2}+\mathcal{R}({{v}_{\boldsymbol{b}}}(t_0))+G({{v}_{\boldsymbol{b}}}(t_0))\geqslant C(a_0),
		\end{align*}
		where $\mathcal{R}({{v}_{\boldsymbol{b}}}(t_0))=\mathcal{E}_2({{v}_{\boldsymbol{b}}}(t_0)){{v}_{\boldsymbol{b}}}^{(2)}(t_0)>0$ is given by \eqref{signfunction}.
		The last inequality contradicts \eqref{energybound} and the conservation of energy in Proposition~\ref{prop:conservation}.
		In particular, if $v_{\boldsymbol{b}}$ enters the interval $(R(a_0), +\infty)$, then $v_{\boldsymbol{b}}$ cannot leave it again, and hence is certainly not a periodic solution.
		
		\noindent{\bf Claim 3:} $0\in S_1\neq\varnothing$.
		
		\noindent On the other hand, if $\boldsymbol{b}=\boldsymbol{0}$, we see from \eqref{odeclassification1} that $v_{\boldsymbol{0}}^{(6)}(0)=g(a_0)<0$, and hence $v_{\boldsymbol{0}}$, $v^{(2)}_{\boldsymbol{0}}$, and $v^{(4)}_{\boldsymbol{0}}$ are strictly decreasing on some small interval $t \in(0, \tau)$ for some $\tau>0$. Since $g(v)<0$ for $v\in(0, a_0)$, we deduce from \eqref{odeclassification1} that $v_{\boldsymbol{0}}^{(j)}$ for $j=1,\dots,5$ stay strictly negative until $v_{\boldsymbol{0}}$ reaches a negative value. 
		Therefore, if $\boldsymbol{b}=\boldsymbol{0}$, there must be $t_{0}$ such that $v_{\boldsymbol{0}}(t_{0})<0$.
		
		Since our shooting parameter interval $\mathbb R^{2}_+$ is connected, we deduce that $S_1 \cup S_2 \neq \mathbb R^{2}_+$. 
		Hence there exists $\boldsymbol{b}^{*}\in\mathbb R^{2}_+\setminus\{0\}$ and a corresponding solution $v^{*}:=v_{\boldsymbol{b}^{*}}$ such that $0 \leqslant v^{*} \leqslant R(a_0)$, which means that $v^{*}$ is bounded. 
		Combining this with the fact that $g$ is locally Lipschitz, we get that $T_{\boldsymbol{b}^{*}}=\infty$. 
		By even reflection, we find a solution defined on $\mathbb{R}$, which we still refer to as $v^{*}$. 
		Since $\boldsymbol{b}^{*}\in\mathbb R^{2}_+\setminus\{0\}$, we know $v^{*}$ has a strict local minimum at the origin. 
		Therefore, using the classification from Lemma~\ref{lm:classification}, $v^{*}$ must be periodic. Moreover, it has a unique local maximum and minimum per period and is symmetric with respect to its extrema.
		The uniqueness of $v^{*}$ up to translations follows from Lemma~\ref{lm:energyordering}.
		
		\noindent{\bf Step 3.} Proof of \ref{itm:C3}.
		
		Using that $\lambda_1=\gamma_n$, we need to prove that $w:=v^{(1)}/v$ satisfies $w(t)<\lambda_1$ for all $t\geqslant0$. Using Lemma~\ref{lm:classification}, one can find $t_{0} \in \mathbb{R}$ such that $v^{(1)}(t_{0})=0$, and so $w(t_{0})=0$. Then, we reduced to prove the following claim
		
		\noindent{\bf Claim 4.} $M:=\{t>t_{0}: w(t) \geqslant \lambda_1\}=\varnothing$.
		
		\noindent Indeed, notice that $w\in C^{5}(\mathbb R)$ satisfies
		\begin{equation}\label{quotientode}
			w^{(1)}=-w^{2}+\lambda_1+\frac{\Phi_1}{v},
		\end{equation}
		where $\Phi_1=\partial_t^{(2)}v-\lambda_1v$.
		Now, suppose by contradiction that $M \neq \varnothing$ and let $t_{1}:=\inf M$. 
		It is easy to see that $t_{1}>t_{0}$, which yields $w^{(1)}(t_{1}) \geqslant 0$. 
		On the other hand, since $w(t_{1})=\lambda_1$, \eqref{quotientode} implies
		\begin{equation*}
			w^{(1)}(t_{1})=\frac{\Phi_1(t_{1})}{v(t_{1})},
		\end{equation*}
		which combined with $v(t_1)>0$ yields
		\begin{equation}\label{contradiction}
			\Phi_1(t_{1})>0.
		\end{equation}
		
		Next, using the notation in the proof of Lemma~\ref{lm:comparisonprinciple}, we can rewrite \eqref{ourODE} as 
		\begin{equation*}
			-L_{\lambda_1\lambda_2}(\Phi_{2}):=-\Phi_{2}^{(2)}+\lambda_1\lambda_2\Phi_{2}=-f(v) \quad {\rm in} \quad \mathbb R,
		\end{equation*}
		where $\Phi_{1},\Phi_{2},\Phi_{3}$ are defined in \eqref{auxiliaryfunctions}.
		According to Lemma~\ref{lm:classification}, $v$ attains its maximum in $\mathbb{R}$. Since $v>0$, the maximum principle implies that $\Phi_{2}<0$ in $\mathbb{R}$. 
		Thus, since
		\begin{equation*}
			-L_{\lambda_1}(\Phi_{1}):=-\Phi_{1}^{(2)}+\lambda_1\Phi_{1}=\Phi_{2} \quad {\rm in} \quad \mathbb R,
		\end{equation*}
		the maximum principle again implies that $\Phi_{1}<0$ in $\mathbb{R}$.
		This is a contradiction with \eqref{contradiction}, which finishes the proof.
	\end{proof}
	
	\section{Proof of the main theorem}\label{sec:proofofmaintheorem}
	In this section, we are based on Proposition~\ref{prop:radialsymmetry} and Proposition~\ref{prop:odeclassification} to prove Theorem~\ref{maintheorem}.
	
	\begin{proof}[Proof of Theorem~\ref{maintheorem}]
		First, using Proposition~\ref{prop:radialsymmetry}, it follows that $u$ is radially symmetric.
		Notice that $v(t)=e^{\gamma_{n}t}u(e^{t})$ satisfies \eqref{ourODE}, we are in position to apply the classification result from Proposition~\ref{prop:odeclassification}.
		Hence either $v\equiv a^{*}_n$ is constant or $v$ is periodic. 
		Indeed, the only case that remains to be excluded is that $v(t)=c_{n}(2 \cosh (t-T))^{\gamma_{n}}$ for some $T\in\mathbb R$. 
		However, if this holds, we would have that $v(t) \sim c_{n} e^{\gamma_{n}t}$ as $t \rightarrow-\infty$ and hence the singularity of $u$ would be removable, which is a contradiction since $0$ is a non-removable singularity. 
		Thus, either $v$ is constant or periodic.
		
		Second, fixing $a_0:=\inf v$ and using Proposition~\ref{prop:odeclassification} (i), it follows that $a_0 \in(0, a^{*}_n]$, and $a_0=a^{*}_n$ if and only if $v \equiv a^{*}_n$. 
		Moreover, when $a_0\in(0,a^{*}_n$), the function $v$ is periodic with minimal value $a_0$. Therefore, by the Proposition~\ref{prop:odeclassification} (ii), we get that $v(t)=v_{a}(t+T)$ for some $T \in \mathbb{R}$.
		
		Finally, a simple computation shows in Emden--Fowler coordinates that $\partial^{(1)}_ru<0$ in $(0,+\infty)$ is equivalent to $v^{(1)}<\gamma_{n} v$ in $\mathbb R$, which is true because of 
		Proposition~\ref{prop:odeclassification} (iii).
		This concludes the proof of our main theorem.
	\end{proof}
	
	\begin{acknowledgement}
		This paper was finished when the first-named author held a Post-doctoral position at the University of British Columbia, whose hospitality he would like to acknowledge.
		He also wishes to express gratitude to Professor Jo\~ao Marcos do \'O for his constant support and several valuable conversations.
	\end{acknowledgement}
	

\end{document}